\documentclass{iopart}

\usepackage{t1enc}
\usepackage{graphicx}
\usepackage[english]{babel}
\usepackage[T1]{fontenc}
\usepackage[latin1]{inputenc}
\usepackage{url}
\usepackage{graphicx}
\usepackage{graphics}
\usepackage{subfig}
\usepackage{iopams} 
\usepackage{setstack}
\usepackage{amstext}
\usepackage{amsopn}
\usepackage{amsthm}

\newtheorem{defi}{Definition}
\newtheorem{theo}{Theorem}

\newcommand{\R}{\mathbb{R}}
\newcommand{\J}{\mathcal{J}}
\newcommand{\norm}[1]{\Vert #1 \Vert}
\newcommand{\sqnorm}[1]{\Vert #1 \Vert_2^2}
\newcommand{\linorm}[1]{\Vert #1 \Vert_{\Sigma_\varepsilon^{-1}}^2}
\newcommand{\abs}[1]{\vert #1 \vert}
\newcommand{\sign}{\text{sign}}

\newcommand{\argmaxsub}[1]{\underset{{ #1 }}{{\rm argmax}}}
\newcommand{\argminsub}[1]{\underset{{ #1 }}{{\rm argmin}}}
\newcommand{\Exp}{\mathbb{E}}

\newcommand{\uMAP}{\hat{u}_{\text{\tiny MAP}}}
\newcommand{\uCM}{\hat{u}_{\text{\tiny CM}}}
\newcommand{\pMAP}{\hat{p}_{\text{\tiny MAP}}}
\newcommand{\pCM}{\hat{p}_{\text{\tiny CM}}}
\newcommand{\PsiLS}{\Psi_{\text{\tiny LS}}(u,\hat{u})}
\newcommand{\PsiBrg}{\Psi_{\text{\tiny Brg}}(u,\hat{u})}
\newcommand{\PsiLSnoarg}{\Psi_{\text{\tiny LS}}}
\newcommand{\PsiBrgnoarg}{\Psi_{\text{\tiny Brg}}}
\newcommand{\eqab}[1]{\stackrel{ #1 }{=}}
\newcommand{\NoiSig}{\Sigma_{\varepsilon}^{-1}}
\begin{document}

\title[MAP Estimates are Proper Bayes Estimators]{Maximum-A-Posteriori Estimates in Linear Inverse Problems with Log-concave Priors are Proper Bayes Estimators}

\author{Martin Burger$^{1,2}$, Felix Lucka$^{1,2,3}$}

\address{$^1$ Institute for Computational and Applied Mathematics, University of M\"unster, Einsteinstr. 62, D-48149 M\"unster, Germany}
\address{$^2$ Cells in Motion Cluster of Excellence, University of M\"unster, Mendelstr. 12, D-48149 M\"unster, Germany}
\address{$^3$ Institute for Biomagnetism and Biosignalanalysis, University of M\"unster, Malmedyweg 15, D-48149 M\"unster, Germany}

\ead{martin.burger@wwu.de}

\begin{abstract}
A frequent matter of debate in Bayesian inversion is the question, which of the two principle point-estimators, the maximum-a-posteriori (MAP) or the conditional mean (CM) estimate is to be preferred. As the MAP estimate corresponds to the solution given by variational regularization techniques, this is also a constant matter of debate between the two research areas. Following a theoretical argument - the Bayes cost formalism - the CM estimate is classically preferred for being the Bayes estimator for the mean squared error cost while the MAP estimate is classically discredited for being only asymptotically the Bayes estimator for the uniform cost function. \\
In this article we present recent theoretical and computational observations that challenge this point of view, in particular for high-dimensional sparsity-promoting Bayesian inversion. Using Bregman distances, we present new, proper convex Bayes cost functions for which the MAP estimator is the Bayes estimator. We complement this finding by results that correct further common misconceptions about MAP estimates. In total, we aim to rehabilitate MAP estimates in linear inverse problems with log-concave priors as proper Bayes estimators. 
\end{abstract}

\submitto{\IP}
\maketitle


\section{Introduction} \label{sec:Intro}

Bayesian models have received considerable attention in inverse problems over the last years. A particular advantage of the Bayesian approach is the systematic treatment of stochastic forward models and of prior knowledge about solutions, which is closely related to classical regularization theory. While the basic idea is now widely accepted in the inverse problems community, there is still a debate concerning the choice of point estimates. While pragmatical and computational reasons are clearly favouring maximum a-posteriori probability estimates, those are considered inferior to others like conditional mean estimates by statistical arguments. In particular the Bayes cost approach argues the latter to minimize a natural cost while the maximum a-posteriori probability estimate can only be obtained asymptotically from a degenerate cost. In this paper, we will present a novel viewpoint on maximum a-posteriori probability estimates, having in mind high-dimensional log-concave priors such as popular sparsity priors. Our computational and theoretical results puts the inferiority compared to conditional means estimates under question. \\
We consider the inverse problem of solving a linear, ill-posed operator equation for the true, infinite-dimensional solution $\tilde{u}$. Here, we start from the following discrete model chosen for obtaining a computational solution:
 \begin{equation}
  f = K \, u + \varepsilon, \label{eq:FwdEq1}
 \end{equation}
where $f \in \R^m$ represents the given measured data, $u \in \R^n$ represents a discretization of $\tilde{u}$, $ K \in \R^{m \times n}$ is the discretization of the continuous forward operator with respect to the spaces of $u$ and $f$ and $\varepsilon  \in \R^m$ is an additive, stochastic noise term. For simplicity we restrict ourselves to the case of $\varepsilon$ being Gaussian. We want to solve \eref{eq:FwdEq1} in the framework of Bayesian inversion, which we will briefly sketch in the following (cf. \cite{KaSo05} for further details and \cite{HaLaLeSaTa04,CuFoSu11,KaFo11,HmKaKoLaNiSi13,LaSeKaVaKoToMa13,TaPuCoKaAr13} for exemplary applications):\\
 First, the stochastic nature of the noise term renders \eref{eq:FwdEq1} into a relation between the \emph{random variables} $F$ and $\mathcal{E}$:
\begin{equation}
 F =   K \,  u + \mathcal{E}, \label{eq:LikeMod}
\end{equation}
where we assume $\mathcal{E} \sim \mathcal{N}(0,\Sigma_\varepsilon)$. Now, \eref{eq:LikeMod}  determines the conditional probability density of $F$ given $u$ (the \emph{likelihood} density):
\begin{equation}
  p_{li}(f|u) \propto \exp \left( -\frac{1}{2 } \linorm{ f - K \, u }\right), \quad \text{with} \quad  \norm{y}_A^2 := y^T A \, y\label{Likelihood}
\end{equation}
Standard statistical inference strategies like \emph{maximum-likelihood estimation} would try to estimate $u$ on the basis of \eref{Likelihood}. However, in typical inverse problems, the ill-posedness of \eref{eq:FwdEq1} precludes this approach. \emph{Bayesian inference strategies} rely on encoding \emph{a-priori} information about $u$ by modeling it as a random variable as well ($U$ in our notation). Its density, $p_{pr}(u)$ is therefore called the \emph{prior}. \emph{Bayes' rule} can then be used to construct the posterior probability density:
\begin{equation}
  p_{po}(u|f) = \frac{p_{li}(f|u)p_{pr}(u)}{p(f)} \label{eq:BayesRule}
\end{equation}
This conditional density of $U$ given $F$ is called the \emph{posterior}
In Bayesian inversion, this density is the complete solution to the inverse problem and, thus, the central object of interest (see Figure \ref{fig:DistPlots} for an illustration). \emph{Bayesian inference} is the process of extracting the information of interest from the posterior:
  \begin{itemize}
  \item Point estimates infer a single estimate of $u$ from the posterior.
  \item Credible regions estimates search for sets that bound $u$ with a certain probability.
  \item Extreme value probabilities try to estimate the probability that a feature $g(u)$ exceeds some critical value.
  \item Conditional covariance estimates try to assess the spatial distribution of variance and dependencies between the components of $u$.
  \item Histogram estimates analyze the distributions of single components $u_i$.
 \end{itemize}
More advanced Bayesian techniques like the treatment of nuisance parameters by \emph{marginalization} or \emph{approximation error modeling} \cite{KaSo05,KaSo07,NiKoKa11,LiKoRoKo13,TaPuCoKaAr13}, \emph{model comparison, selection or averaging} \cite{TrAuVa04,HeMaPhFr09}, and \emph{experimental design} \cite{To11}
are also based on the above formalism and principles.\\
Up to now, we did not address the most important step in Bayesian inversion, i.e., the construction of the prior $p_{pr}(u)$ (\emph{Bayesian modeling}). All theorems developed in this work are valid for \emph{log-concave Gibbs distributions} of the form
\begin{equation}
  p_{pr}(u) \propto \exp \left(- \lambda \J(u) \right),  \label{eq:GibbsPrior}
\end{equation}
where $\J(u)$ is a convex functional (called the prior \emph{energy}), and $\lambda > 0$ is a scaling parameter. This includes a wide range of distributions commonly used in Bayesian inversion. The corresponding posterior is given by:
\begin{equation}
p_{po}(u|f) \propto \exp\left(-\frac{1}{2 } \linorm{  f - K \, u } - \lambda \J (u) \right) \label{eq:Post}
\end{equation}
We will need some further, but rather technical properties later, which are fulfilled for all commonly used convex $\J(u)$.

 \begin{figure}[tb]
\centering
\subfloat[][Gaussian likelihood \label{subfig}]{\includegraphics[width=0.32\textwidth]{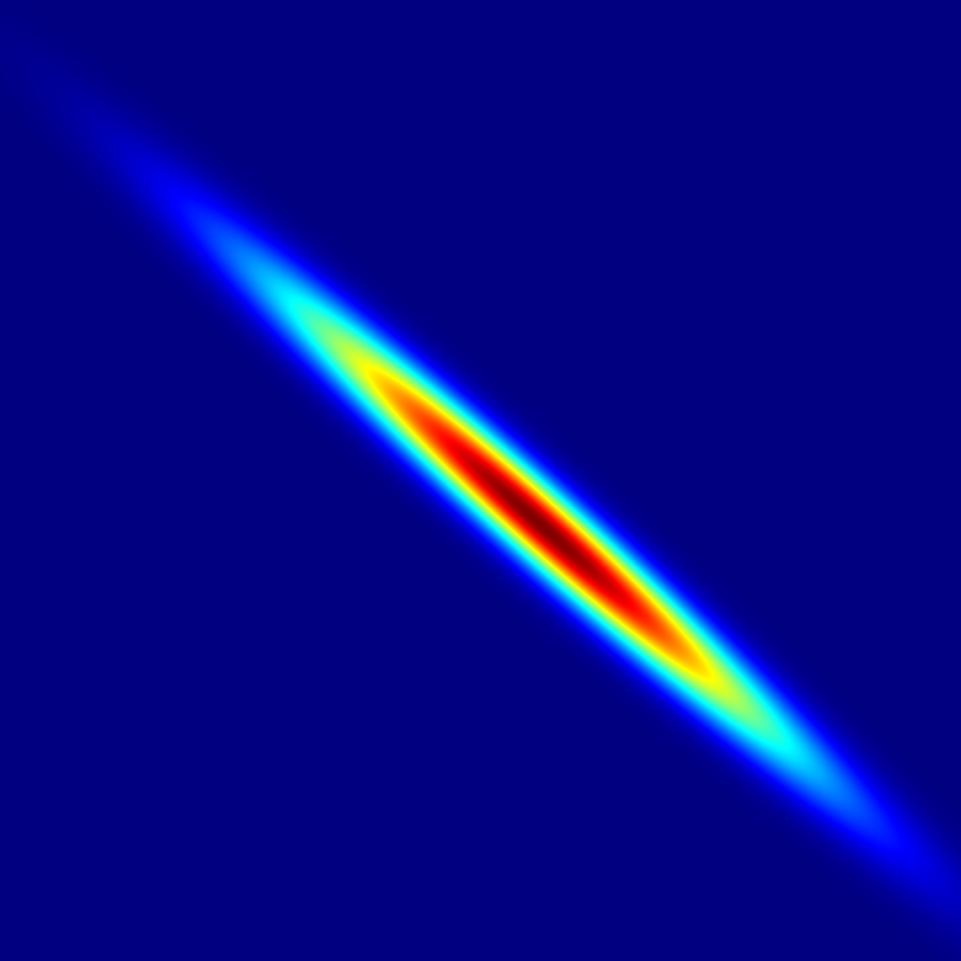}}
\hspace{0.01\textwidth}
\subfloat[Prior: $p(u) \propto \exp (-\lambda | u|_1)$]{\includegraphics[width=0.32\textwidth]{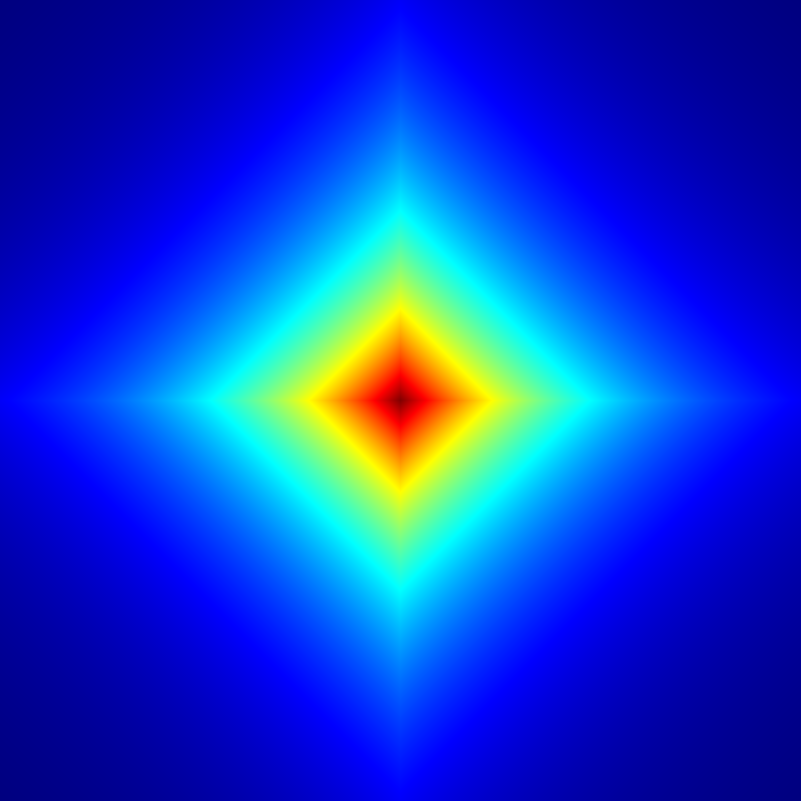}}
\hspace{0.01\textwidth}
\subfloat[Resulting posterior]{\includegraphics[width=0.32\textwidth]{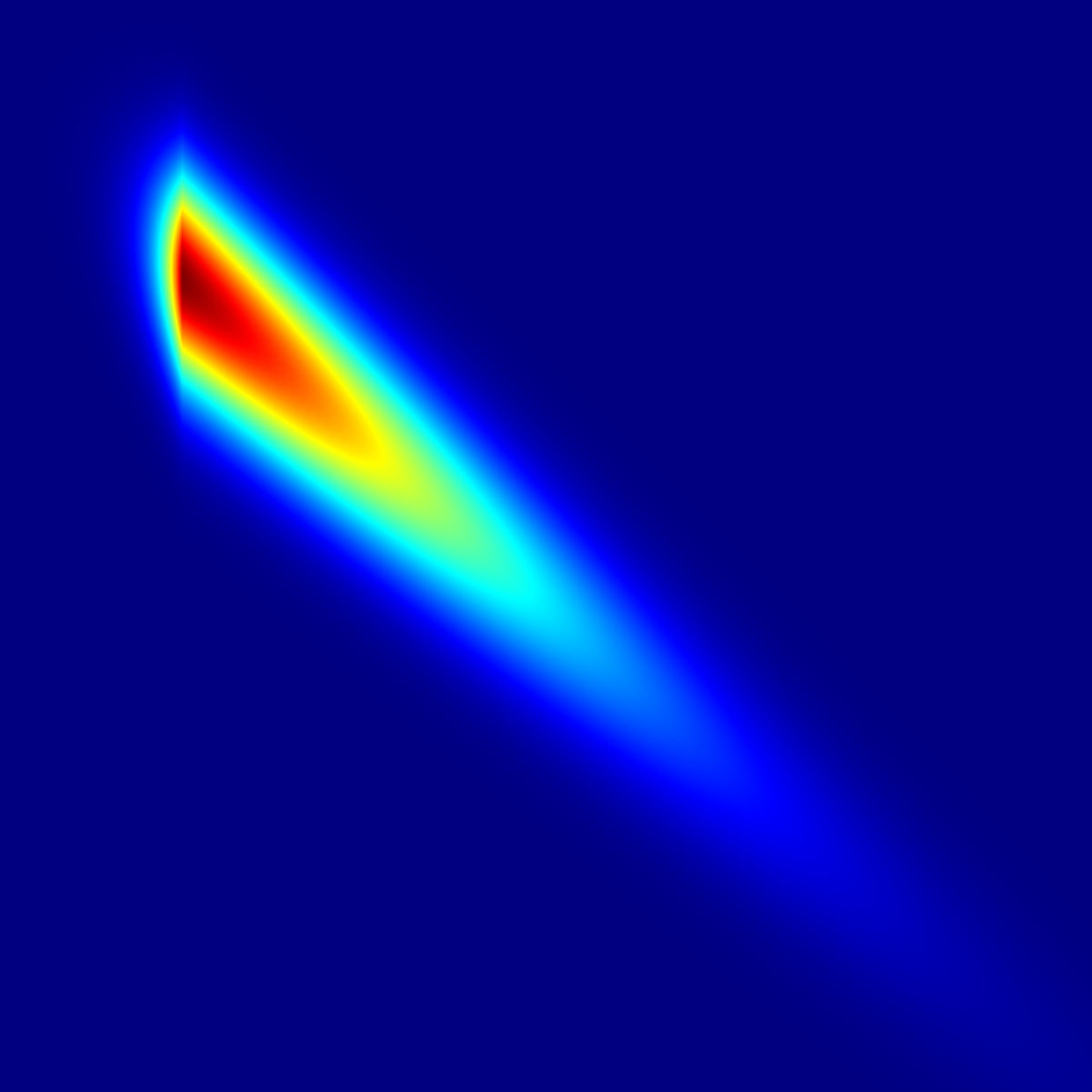}}
\caption{Illustration of possible shapes of likelihood, prior and posterior for $n= 2$.\label{fig:DistPlots}}
\end{figure}

\subsection{MAP vs. CM Estimates: Variational Regularization vs. Bayesian Inference?} \label{subsec:MAPvsCMIntro}
Choosing a single point estimate for $u$ is the most simple but also most commonly used Bayesian inference technique. Two popular estimates are:
   \begin{itemize}
    \item Maximum a-posteriori-estimate (MAP):
      \begin{equation}
      	\uMAP := \argmaxsub{u \in \R^n} \left\lbrace   \; p_{po}(u|f) \right\rbrace \label{eq:MAPDef}
      \end{equation}
	 We can compute $\uMAP$ for \eref{eq:Post} by solving a high-dimensional optimization problem:
	 \begin{equation}
	 \uMAP = \argminsub{u \in \R^n}\left\lbrace \frac{1}{2 } \linorm{  f - K \, u} + \lambda \J (u)  \right\rbrace \label{eq:GenTikh}
	\end{equation}	  
	This is a \emph{Tikhonov-type} regularization of equation \eref{eq:FwdEq1} \cite{EnHaNe96}. Hence, MAP estimation yields a direct correspondence to \emph{variational regularization} techniques \cite{BuOs04,ScKaHoKa12}.
    \item Conditional mean-estimate (CM): 
      \begin{equation}
		\uCM := \Exp \left[ u|f \right] = \int u \; p_{po}(u|f) \; \rmd u \label{eq:CMDef}
      \end{equation}
      Computing $\uCM$ requires solving a high-dimensional \emph{integration} problem \cite{KaSo05,St10}. 
  \end{itemize}
The immediate and obvious question is: What is the difference between MAP and CM estimate? Which of them is ''better'' in general, or for a specific task? This is not only a matter of constant debate within the field of Bayesian inversion, but also with classical regularization theory due to the direct correspondence of $\uMAP$. This article starts with a summary of the ''classical'' view on the issue. Then, several recent computational and theoretical results are discussed, which challenge this point of view. In the last part, new theoretical ideas are introduced that fit to all of these results, disprove certain common myths and will lead to new insights and perspectives for the comparison of variational regularization and Bayesian inference.

\subsection{Sparsity Constraints in Inverse Problems} \label{subsec:SpIP}
\emph{Sparsity constraints} are a type of a-priori information that demand the solution of \eref{eq:FwdEq1} to have very few non-zero coefficients in a suitable representation (i.e., a bases, frames or other dictionaries). Solving high-dimensional inverse problems using sparsity constraints has led to enormous advances in various areas, a popular example being \emph{total variation} (\emph{TV}) deblurring \cite{BuOs13}, based on sparsity constraints on the gradient of the unknown quantity. Commonly, sparsity constraints are formulated in the framework of variational regularization by  choosing $\ell_1$-type norms for constructing $\J(u)$ in \eref{eq:GenTikh}:
\begin{equation}
 \hat{u}_{\alpha} = \argminsub{u \in \R^n}\left\lbrace \frac{1}{2 } \linorm{  f - K \, u } + \lambda \abs{\Phi(u)}_1,  \right\rbrace, \label{eq:SparseTikh}
\end{equation}
where $\Phi(u)$ is a convex function mapping $u$ onto the potentially sparse property, e.g., the local $\ell_2$-norm of its gradient in TV. Recently, using similar sparsity-constraints in Bayesian inversion has attracted considerable attention. There are two common ways to encode sparsity in the prior:
 \begin{enumerate}
  \item Converting the functionals used in \eref{eq:SparseTikh} directly into priors of the form $p_{pr}(u) \propto \exp(-\lambda \abs{\Phi(u)}_1)$ (\emph{$\ell_1$-type priors}). This is convenient, since the prior is log-concave and one already knows that the MAP estimate will be sparse. In this article, we will only present computational results for this type of priors (the theoretical results are, however, valid for all log-concave priors). 
  \item Hierarchical Bayesian modeling (HBM) extends the prior model by an additional level and imposes sparsity on this level. While good results in various applications were obtained with this approach (e.g., \cite{BaCaSo10,LuPuBuWo12,PuKa13,WaJiYuJi13}), a potential difficulty is that the implicit priors over the unknowns are usually not log-concave.
 \end{enumerate}
Figure \ref{fig:RndDraws} shows random images drawn from a Gaussian, a $\ell_1$-type and a Student's t-distribution (a potential implicit prior encountered in HBM). While the visual impression of the $\ell_1$ random image clearly differs from the Gaussian one, it is by no means sparse in the traditional sense. If it would be the true solution, $\uMAP$ would probably not be able to recover it in a satisfactory way. This is due to misconceptions behind the ''reverse engineering'' approach of turning a sparsifying regularization functional $\J(u)$ into a prior $p_{pr} \propto \exp(- \lambda \J(u))$. This has already been noticed in \cite{Gr11,GrCeDa12}. More generally, it points to an inherent difficulty of defining sparsity in the Bayesian framework in a meaningful, consistent and tractable way, which we cannot address further in this article.\\
The paper is structured as follows: In Section~\ref{sec:ClassView}, we will provide a discussion of the comparison between MAP and CM estimates. First, we will revisit the classical view on the problem in Section~\ref{subsec:BayesCost}, which favors the CM estimate on the basis of a theoretical argument: The Bayes cost formalism. While the CM estimate minimizes the mean squared error as a cost function, the MAP estimate is discredited for  minimizing a binary cost function in its degenerate limit. Then, in Section~\ref{subsec:RecRes}, we summarize recent results and observations, which do not fit in this picture. In particular, these results steam from linear inverse problems with sparsity-promoting priors incorporating $\ell_1$-type norms and motivate the further theoretical investigation in Section~\ref{sec:RehabMAP}. There, we will present new, proper, convex Bayes cost functions for both MAP and CM estimates. Using \emph{Bregman distances}, the MAP estimate will become a proper Bayes estimator and the computational observations will not longer be contradictory to the Bayes cost theory. In addition, we present further results that correct common misconceptions about MAP estimates.

 \begin{figure}[tb]
\centering
\subfloat[][ $p(u) \propto \exp \left( - \frac{1}{2} \| u \|_2^2 \right)$]{\includegraphics[width=0.32\textwidth]{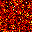}}
\hspace{0.01\textwidth}
\subfloat[][$p(u) \propto \exp\left( -  | u |_1 \right)$]{\includegraphics[width=0.32\textwidth]{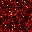}}
\hspace{0.01\textwidth}
\subfloat[][$p(u_i) \propto  (1 + u_i^2/3)^{-2}$]{\includegraphics[width=0.32\textwidth]{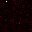}}
\caption{Random images draws from different prior distributions. Here, $u$ was assumed to correspond to a 2D image of $n = 32 \times 32$ pixels $u_1,\ldots,u_{(32^2)}$.\label{fig:RndDraws}}
\end{figure}


\section{MAP vs. CM Estimates} \label{sec:ClassView}

In the following we start by briefly discussing the established viewpoint on the comparison of MAP and CM estimates and subsequently review several recent results converse to this viewpoint, supplemented by some computational experiments. 
 \begin{figure}[tb]
\centering
\includegraphics[height = 4.5cm]{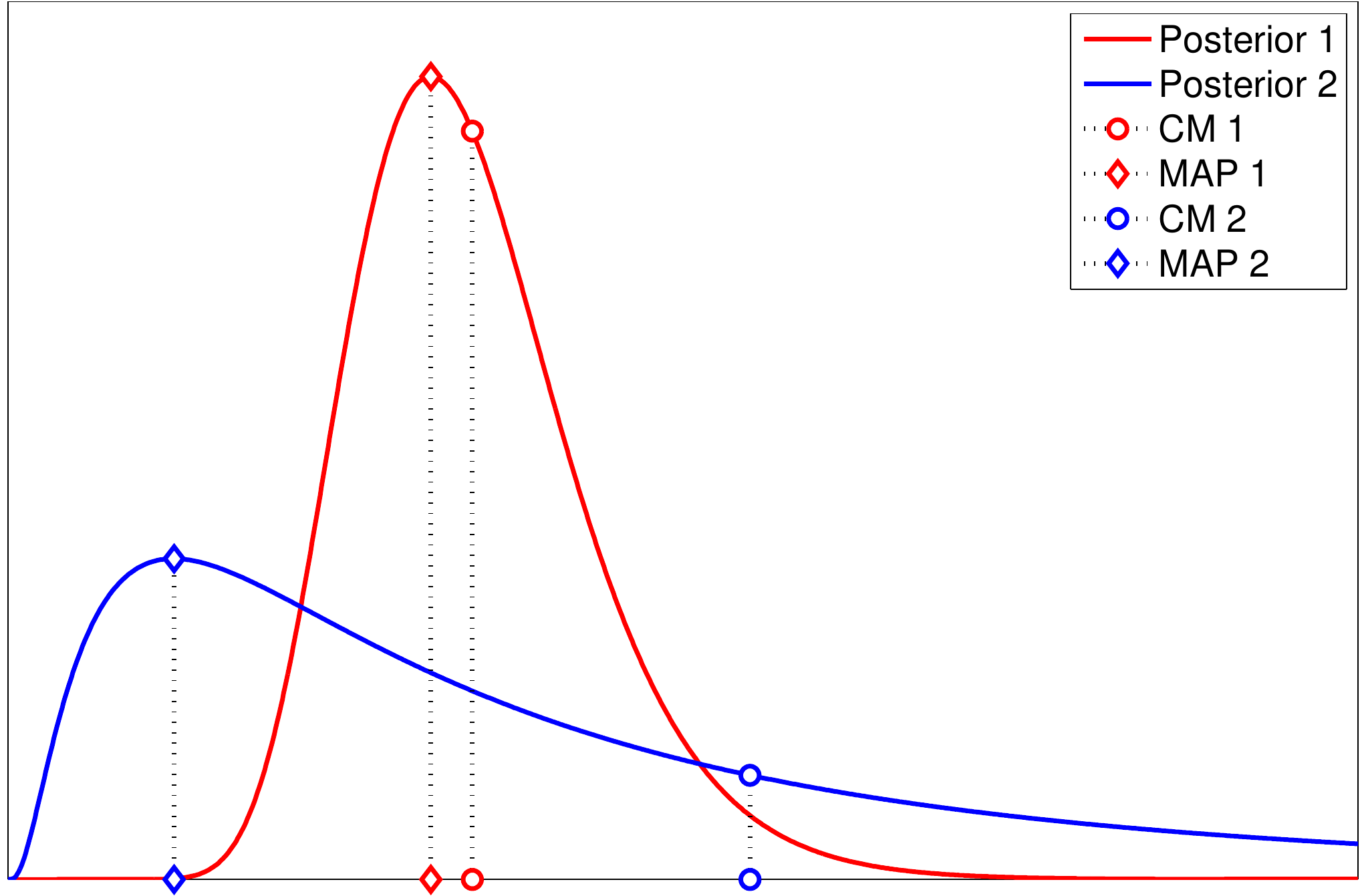}
\caption{Comparison of MAP and CM estimates for two posterior densities.  \label{fig:NaiveCMvsMAP}}
\end{figure}
The CM estimate is the mean of the posterior, while the MAP estimate is the (highest) mode of the posterior (see Figure \ref{fig:NaiveCMvsMAP}). However, this does not provide any intuition  why one of them should be better suited to represent a distribution. Hence, a lot of presentations of the topic provide plots of hypothetical distributions like Figure \ref{fig:PlotsCMvsMAP} to show that none of them is better in general. However, one might argue that the CM estimator as the mean value is an intuitive choice as it is the ''center of (probability) mass'' and corresponds to the average of a sample, familiar from every-day descriptive statistics.
\begin{figure}[tb]
   \centering
   \subfloat[][]{\includegraphics[width=0.49\textwidth]{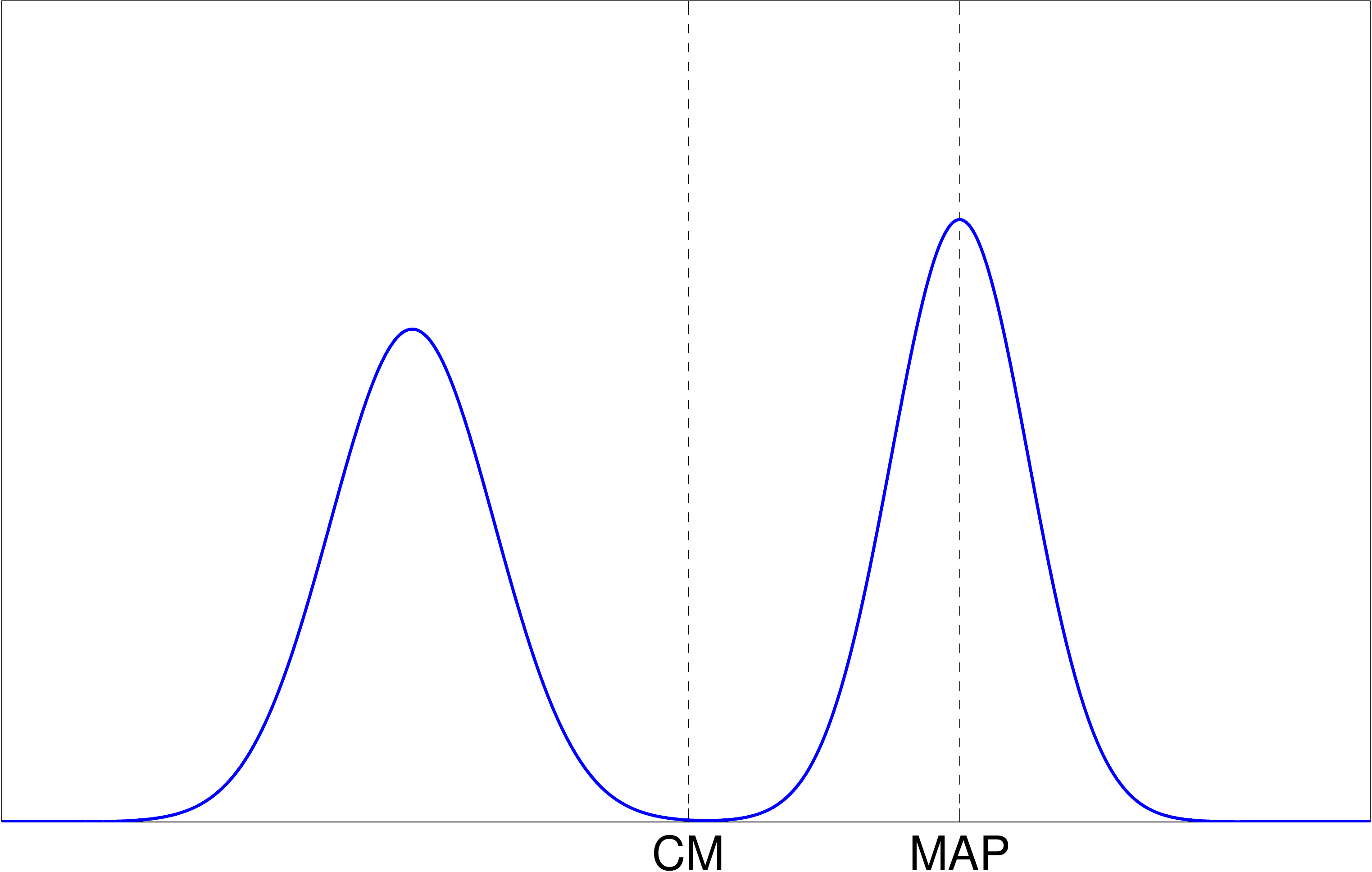}}
   \hspace{0.01\textwidth}
   \subfloat[][]{\includegraphics[width=0.49\textwidth]{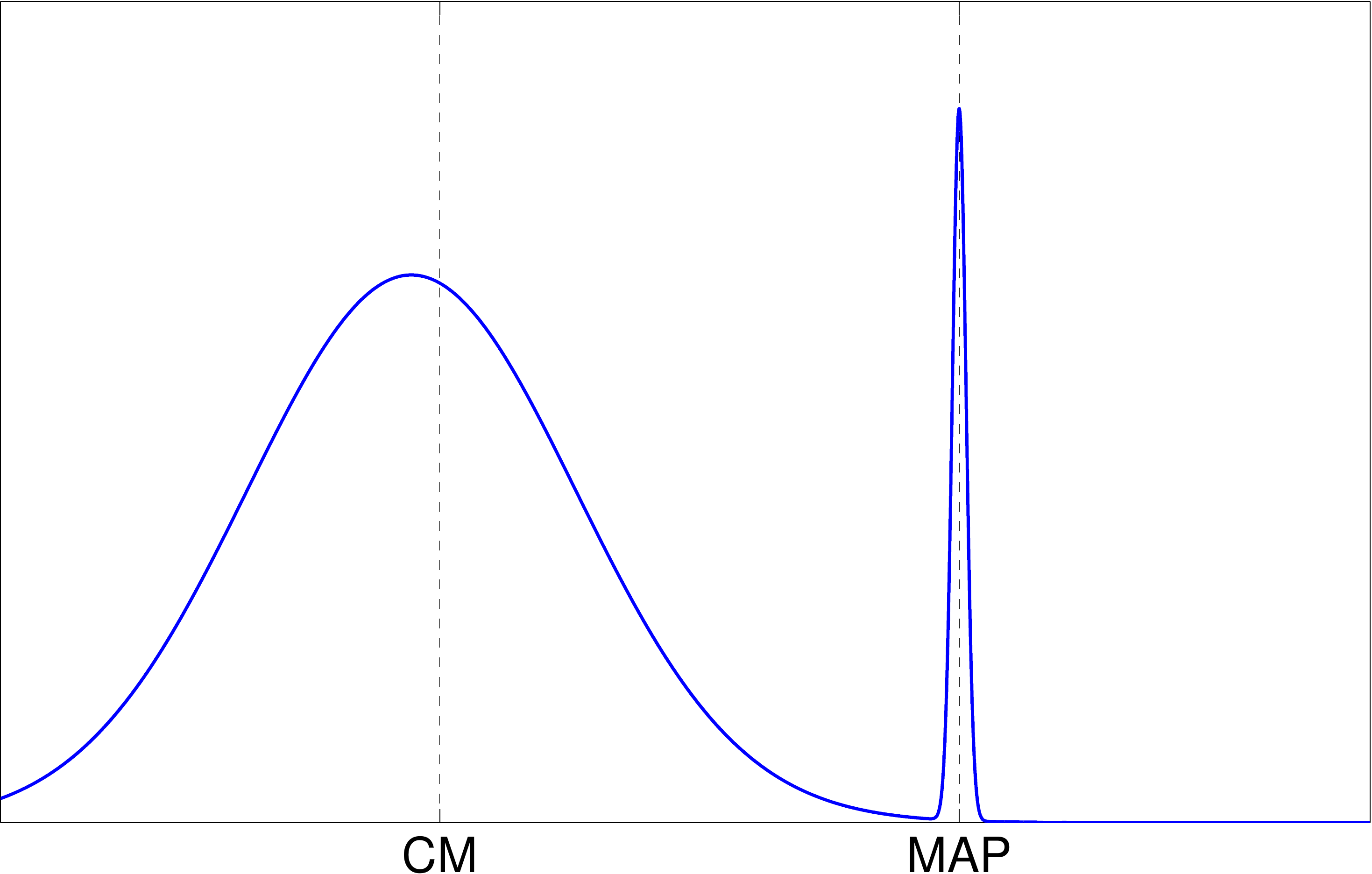}}
   \caption{Hypothetical, bimodal distributions to show that none of the estimates is better in general. \label{fig:PlotsCMvsMAP}}
\end{figure}

 \subsection{Bayes Cost Formalism} \label{subsec:BayesCost}
As the illustrative comparison does not give any useful intuition, the Bayes cost formalism is usually used to provide a decisive theoretical argument, which we recall in the following (cf. \cite{KaSo05,Ka98}). In the Bayesian framework, an estimator $\hat{U}$ is a random variable as well, as it relies on $F$ and $U$. \emph{Statistical estimation theory} (respectively Bayesian decision theory) examines the general behavior of estimators to find optimal estimators for a given task. A common approach is to define a \emph{cost function} $\Psi(u,\hat{u})$ measuring the desired and undesired properties of $\hat{u}$. The \emph{Bayes cost} is defined by the expected cost, i.e., the average performance: 
\begin{eqnarray}
\fl BC_\Psi(\hat{u}) := \Exp\left[ \Psi(u,\hat{u}(f)) \right] &= \int \int \Psi(u,\hat{u}(f)) \: p(u,f) \: \rmd u \: \rmd f \nonumber \\
	&= \int \int \Psi(u,\hat{u}(f)) \: p_{like}(f|u) \: \rmd f \: p_{prior}(u) \: \rmd u \nonumber\\
	&\eqab{\eref{eq:BayesRule}} \int \int \Psi(u,\hat{u}(f)) \: p_{po}(u|f) \: \rmd u \: p(f) \rmd f \label{eq:BayesCost}
\end{eqnarray}
The \emph{Bayes estimator} $\hat{u}_\Psi$ is the estimator, which minimizes $BC_\Psi(\hat{u})$.
\begin{equation*}
\hat{u}_\Psi := \argminsub{\hat{u}} \left\lbrace BC_\Psi(\hat{u}) \right\rbrace
\end{equation*}
In \eref{eq:BayesCost}, $\hat{u}(f)$ only depends on $f$ and the marginal density $p(f)$ is non-negative. Thus, $\hat{u}_\Psi$ also minimizes
\begin{equation}
\hat{u}_\Psi(f) = \argminsub{\hat{u}} \left\lbrace \int \Psi(u,\hat{u}(f)) \: p_{po}(u|f) \: \rmd u \right\rbrace \label{eq:BayesEstimator}
\end{equation}
The main classical arguments in favour of CM and against MAP estimates derived from the Bayes cost formalism are as follows:
\begin{itemize}
 \item The CM estimate is Bayes estimator for the mean squared error
 \begin{equation}
 \Psi_{\text{\tiny MSE}}(u,\hat{u}) = \sqnorm{u - \hat{u}}, \label{eq:MSE}
 \end{equation}
 which seems to be a very natural and reasonable choice for $\Psi$. Interpreted geometrically, one also speaks of a "well-centeredness" of $p_{po}(u|f)$ around $\uCM$. As it is by default unbiased with respect to $p_{po}(u|f)$, one can further show that $\uCM$ is also the \emph{minimum error variance estimator}.
 \item On the other hand, the MAP estimate can only be seen as an \emph{asymptotic} Bayes estimator of
 \begin{equation}
 \Psi_{\delta}(u,\hat{u}) = \cases{0,& if $\|u - \hat{u} \|_\infty \leqslant \delta$\\
1& otherwise\\} \label{eq:UniCost}
\end{equation}
for $\delta \rightarrow 0$ (\emph{uniform cost} or \emph{0-1 loss}). Thus, it is usually not considered a proper Bayes estimator. This characterization also does not seem to allow for an intuitive geometrical interpretation of $\uMAP$ akin to the one for $\uCM$.
\end{itemize}
The theoretical differences between MAP and CM estimates seem to match their practical differences, in particular the different complexity of their computation (cf. Section~\ref{subsec:MAPvsCMIntro}). The theoretical discrimination of the MAP estimate contrasts its success in practical applications, in particular in high-dimensional scenarios. Therefore, one often encounters a strange contrariness in articles about high-dimensional Bayesian inversion: Usually a careful prior modeling is presented, and the CM estimate is regarded as the optimal inference technique. However, for computational reasons, often only a MAP can be computed. This circumstance is usually regretted and excused for. If the computational results are not fully satisfactory, shortcomings of the MAP estimate are discussed as a potential reason for it. However, even if the results are really good, concern is expressed that computing MAP estimates is not a proper Bayesian technique.

\subsection{Converse Results} \label{subsec:RecRes}

In the following we discuss some recent theoretical results and computational experiments indicating that CM estimates are not superior in particular for high-dimensional inversion. 
We start with the Gaussian case: Gaussian priors are the most popular and arguably the most fundamental class of priors one can consider for \eref{eq:Post}, due to various reasons such as their maximum entropy property, alpha-stability and the central limmit theorem. However, for this most fundamental class of priors, the seemingly fundamentally different MAP and CM estimate happen to be equal. From the classical view, this can only be interpreted as a meaningless coincidence, which is arguably not fully satisfactory. One might argue however that a quadratic Bayes cost is perfectly suited for Gaussian models, since it corresponds well to the negative logarithm of the prior. An appropriately scaled square-norm Bayes cost as used above should obviously be quite robust also for high-dimensional problems in the case of Gaussian priors, while this is not clear at all for other priors. For nonquadratic priors asymptotically concentrating on some Banach space it is not at all clear whether there is a robust and asymptotically meaningful squared norm in the Bayes cost criterion, hence it might be reasonable to think about other costs better suited to the Banach space limit. We will further dwell upon this issue in the Section~\ref{sec:RehabMAP}, and continue with results from computational experiments with $\ell_1$ type priors (cf. Section~\ref{subsec:SpIP}). We will use the Split Bregman method \cite{GoOs09} for computing MAP estimates and a specific MCMC scheme we developed in \cite{Lu12} for computing CM estimates.

\subsubsection{A 2D Deblurring Example} \label{subsubsec:2DImDeb}
\begin{figure}[tb]
\centering
\subfloat[][]{
\fbox{\includegraphics[height = 0.45\textwidth]{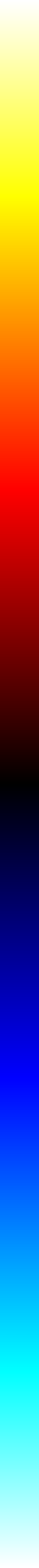}}
}
\hspace{0.01\textwidth}
\subfloat[][Unknown function $\tilde{u}$]{\includegraphics[height = 0.45\textwidth]{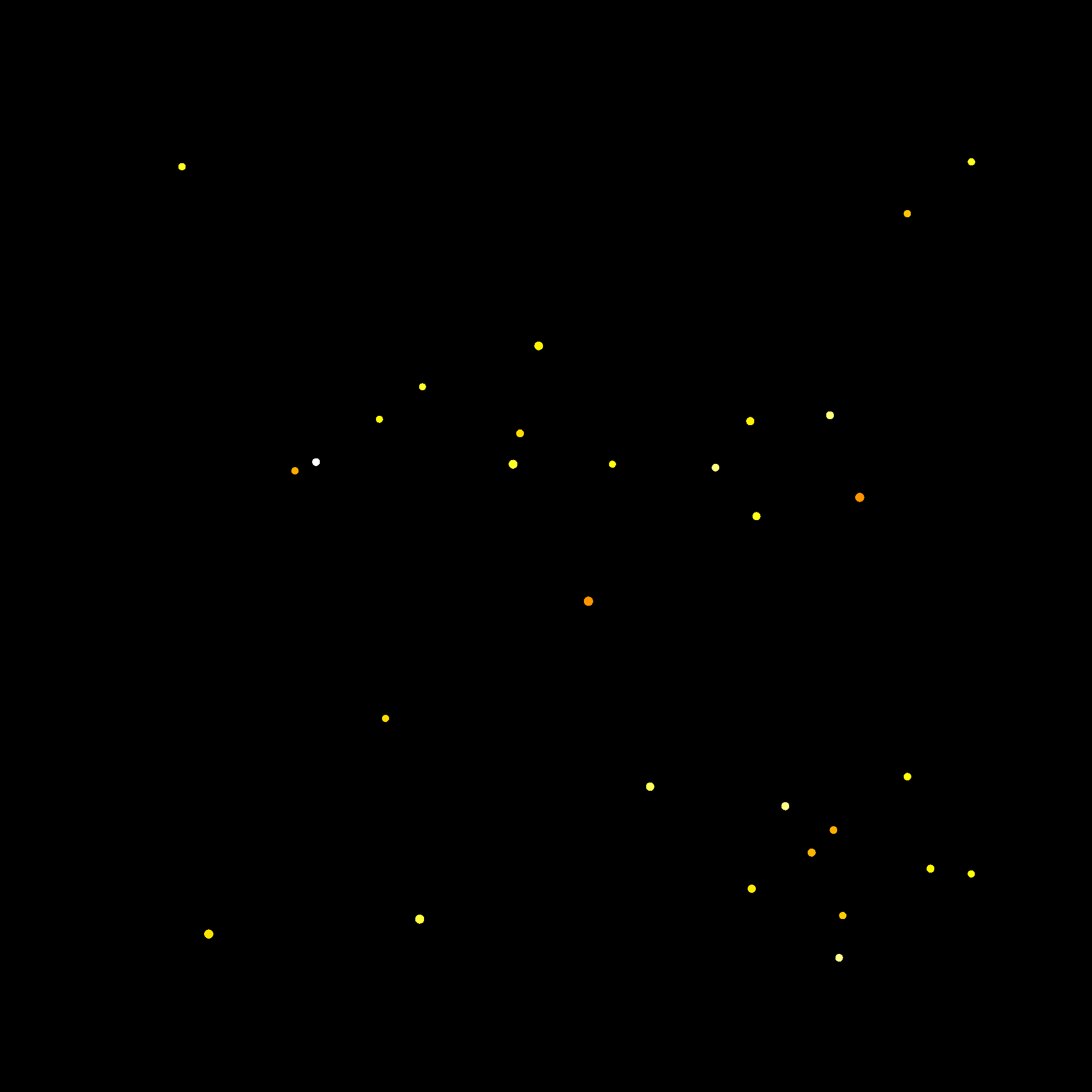} \label{subfig:2Dreal}}
\hspace{0.01\textwidth}
\subfloat[][ Data $f$]{\includegraphics[height = 0.45\textwidth]{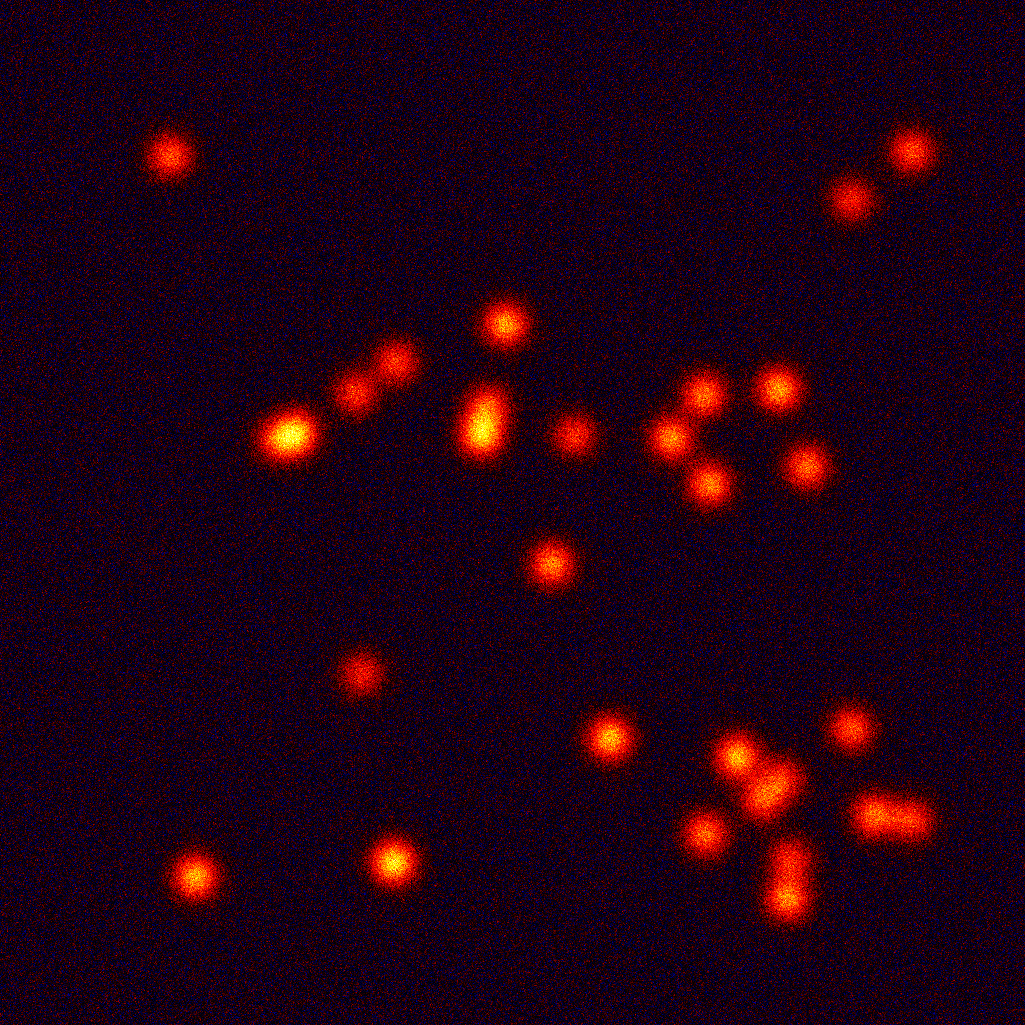} \label{subfig:2Ddata}}
\caption{A simple 2D deblurring example.}
\end{figure}
\begin{figure}[bt]
   \centering
\subfloat[][]{\fbox{\includegraphics[height = 0.45\textwidth]{cool2hot.png}}}
\hspace{0.01\textwidth}
\subfloat[][CM estimate]{\includegraphics[width=0.45 \textwidth]{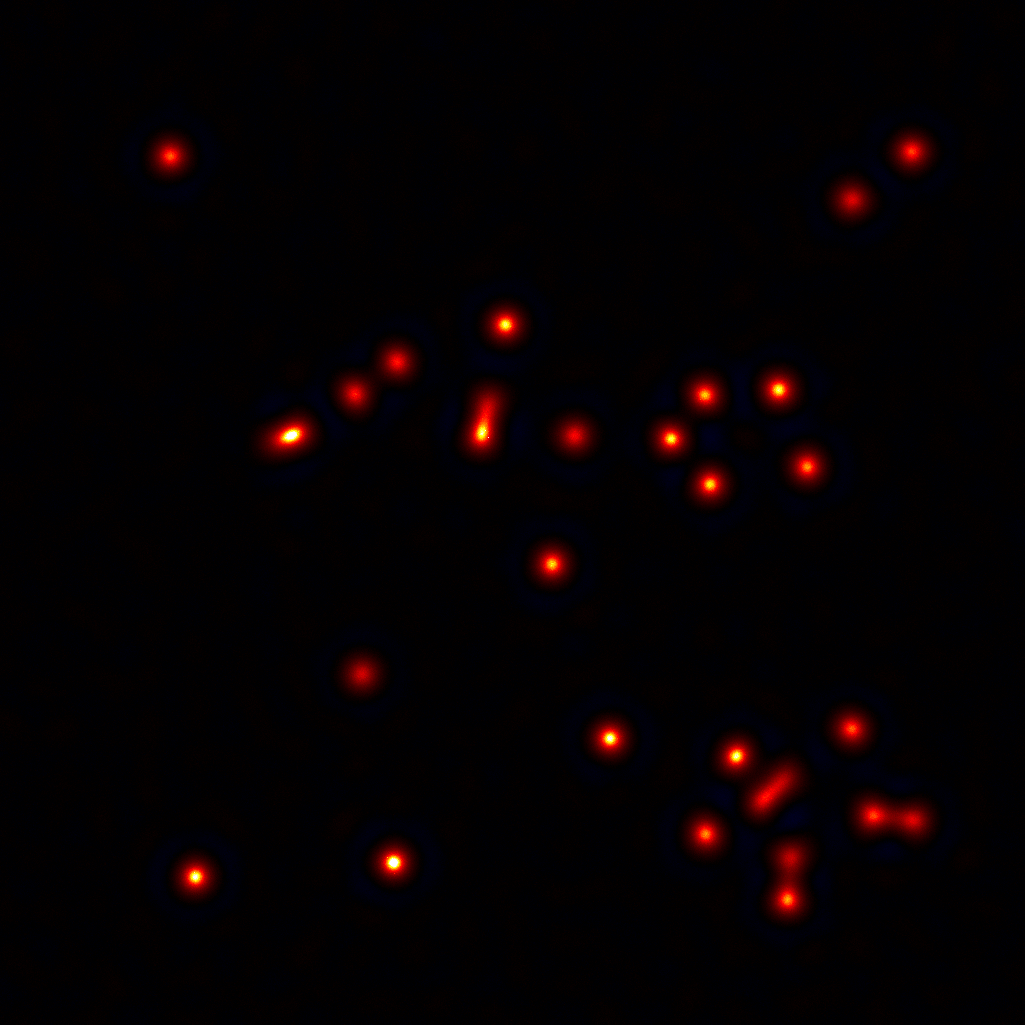}}
\hspace{0.01\textwidth}
\subfloat[][MAP estimate]{\includegraphics[width=0.45 \textwidth]{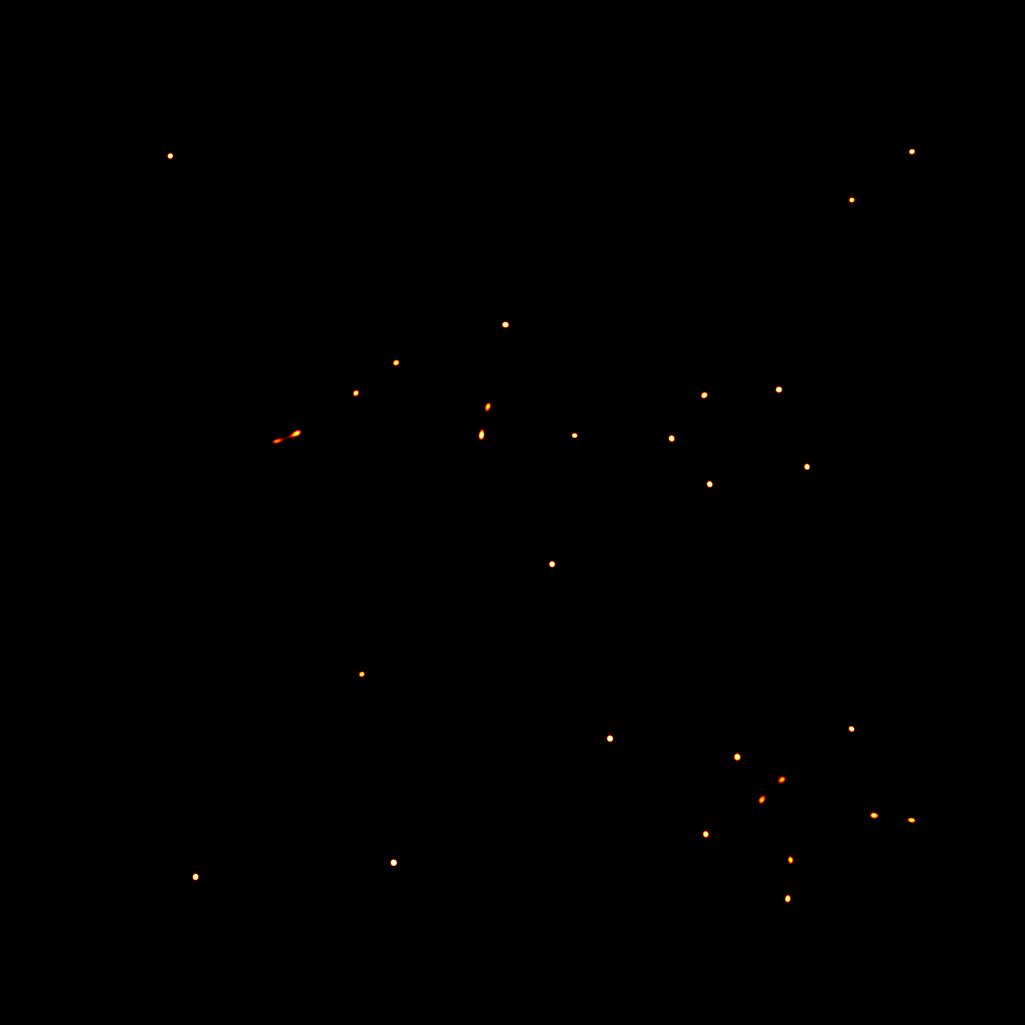}}
\caption{CM and MAP estimate for the 2D deblurring example. \label{fig:2DCMvsMAP}}
\end{figure}
We start with simple 2D image deblurring using the $\ell_1$ prior,  i.e., $p_{pr}(u) \propto \exp (-\lambda |u|_1)$. The unknown intensity function $\tilde{u}:[0,1]^2 \rightarrow \mathbb{R}_+$ consists of circular spots of constant intensity whose radii and intensities slightly vary between single spots (see Figure \ref{subfig:2Dreal}). It is convoluted with a Gaussian kernel (standard deviation of 0.015), integrated over $1025 \times 1025$ regular pixels and contaminated by noise ($\sigma = 0.1 \cdot \norm{K\tilde{u}}_\infty$). The resulting measurement data is displayed in Figure \ref{subfig:2Ddata}. The image will be reconstructed using the $\ell_1$ prior,  i.e., $p_{pr}(u) \propto \exp (-\lambda |u|_1)$ on the same pixel grid used for the measurement. To avoid an inverse crime \cite{KaSo05}, the grid used for the generation of the measurement data was four times finer. The results are shown in Figure \ref{fig:2DCMvsMAP}. While the MAP estimate is very close to $\tilde u$, the CM estimate is blurred and is not able to separate all intensity spots.

\subsubsection{The Discretization Dilemma of the TV Prior}
\begin{figure}[bt]
   \centering
\subfloat[][Unknown function $\tilde{u}$ \label{subfig:TVReal}]{\includegraphics[width = 0.49\textwidth]{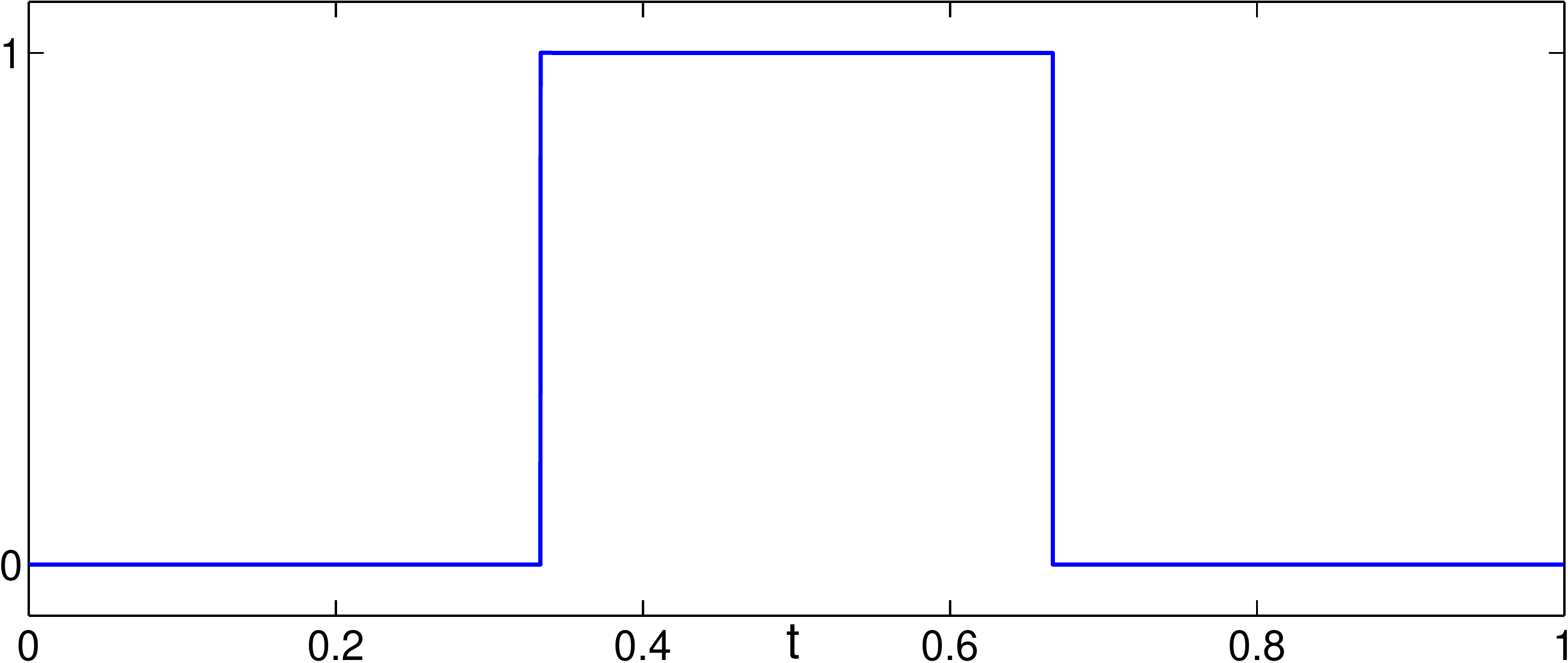}}
\hspace{0.01\textwidth}
\subfloat[][ Data $f$ \label{subfig:TVData}]{\includegraphics[width = 0.49\textwidth]{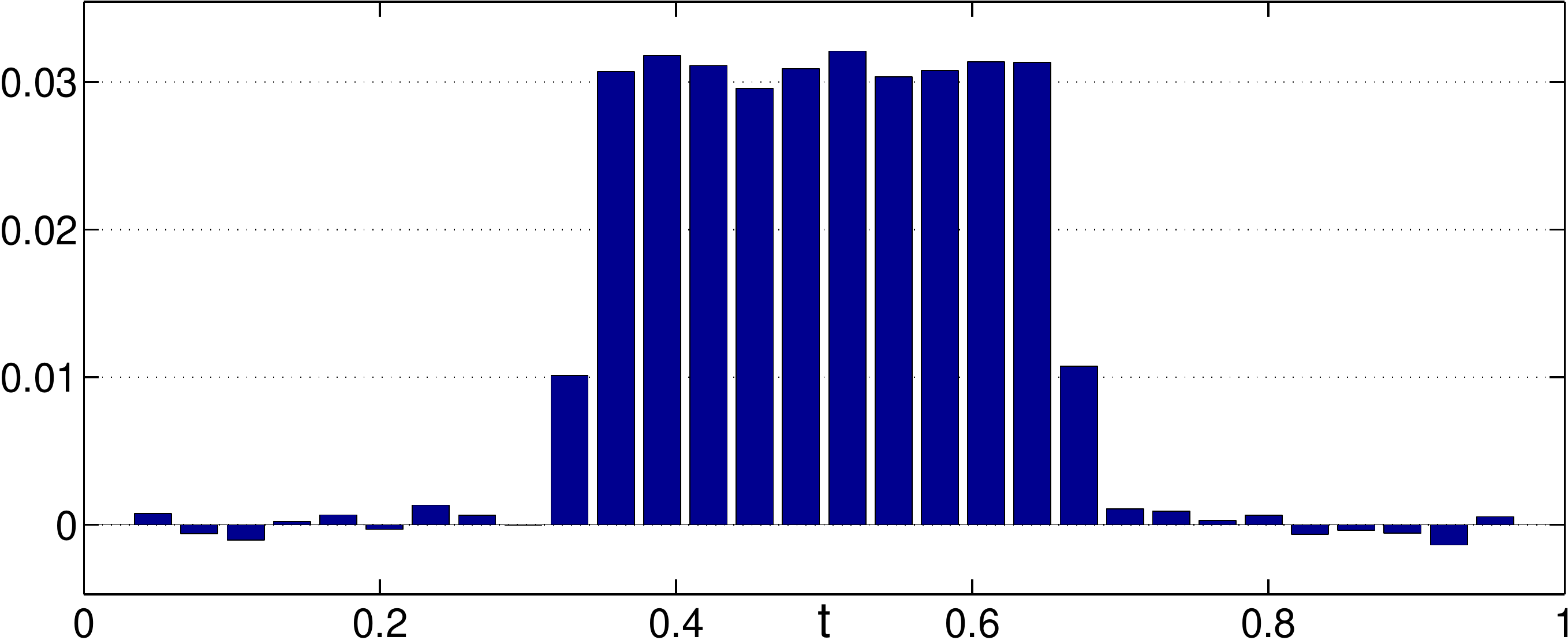}}
\caption{A simple 1D edge-preserving image reconstruction scenario.}
   \label{fig:TVscenario}
\end{figure}

As a second example, we choose a basic 1D scenario to examine \emph{edge-preserving} image reconstruction using the TV prior, i.e., $p_{pr}(u) \propto \exp \left( - \lambda_n \sum_{i=1}^{n-1} |u_{i+1} - u_i| \right)$, where $u_i:= u(t_i)$. We indexed $\lambda_n$ by $n$ to stress that we will choose it depending on the discretization level $n$. The unknown (light intensity) function $\tilde{u}: [0,1] \rightarrow \R_+$ is the indicator function of $[{\small \frac{1}{3}},{\small \frac{2}{3}}]$, see Figure \ref{subfig:TVReal}. It is piecewise integrated over $m$ equidistant intervals (to mimic a measurement by CCD pixels) and contaminated with noise (see Figure \ref{subfig:TVData}). Further details can be found in \cite{Lu12}.\\
This is a toy model for imaging applications where the task is to reconstruct an intensity image that is known to consist of piecewise homogeneous parts with sharp edges( cf. \cite{BuOs13}). In Bayesian inversion, the use of TV priors stimulated interesting developments: In \cite{LaSi04} it was shown that it is not possible to formulate the TV prior in a \emph{discretization invariant} way, i.e., such that the posterior converges to a well defined limit probability density when $n$ is increased while still reflecting the a priori information of edge-preservation. To summarize their results:
\begin{itemize}
\item For $n \rightarrow \infty$ the posterior only converges for $\lambda_n \propto \sqrt{n}$. However, its limit is a Gaussian smoothness prior and the CM estimate converges to a smooth limit while the MAP estimate converges to constant function. This is illustrated in Figure \ref{fig:TVsqrt}, where we computed CM and MAP estimates up to $n = 2^{16}-1$.
\item For $\lambda_n = const.$ and $n \rightarrow \infty$ both posterior and CM estimate diverge, while the MAP estimate converges to an edge-preserving limit, see Figure \ref{fig:TVconst}. Figure \ref{subfig:TVconstZoom} shows a zoom to clarify the divergence of the CM estimate while Figure \ref{subfig:TVconstCheck} demonstrates that this is not an error of the MCMC sampler to compute it by comparing two CM estimates computed from independent MCMC chains. 
\end{itemize}

\begin{figure}[tb]
   \centering
\subfloat[][CM]{\includegraphics[height = 0.49 \textwidth]{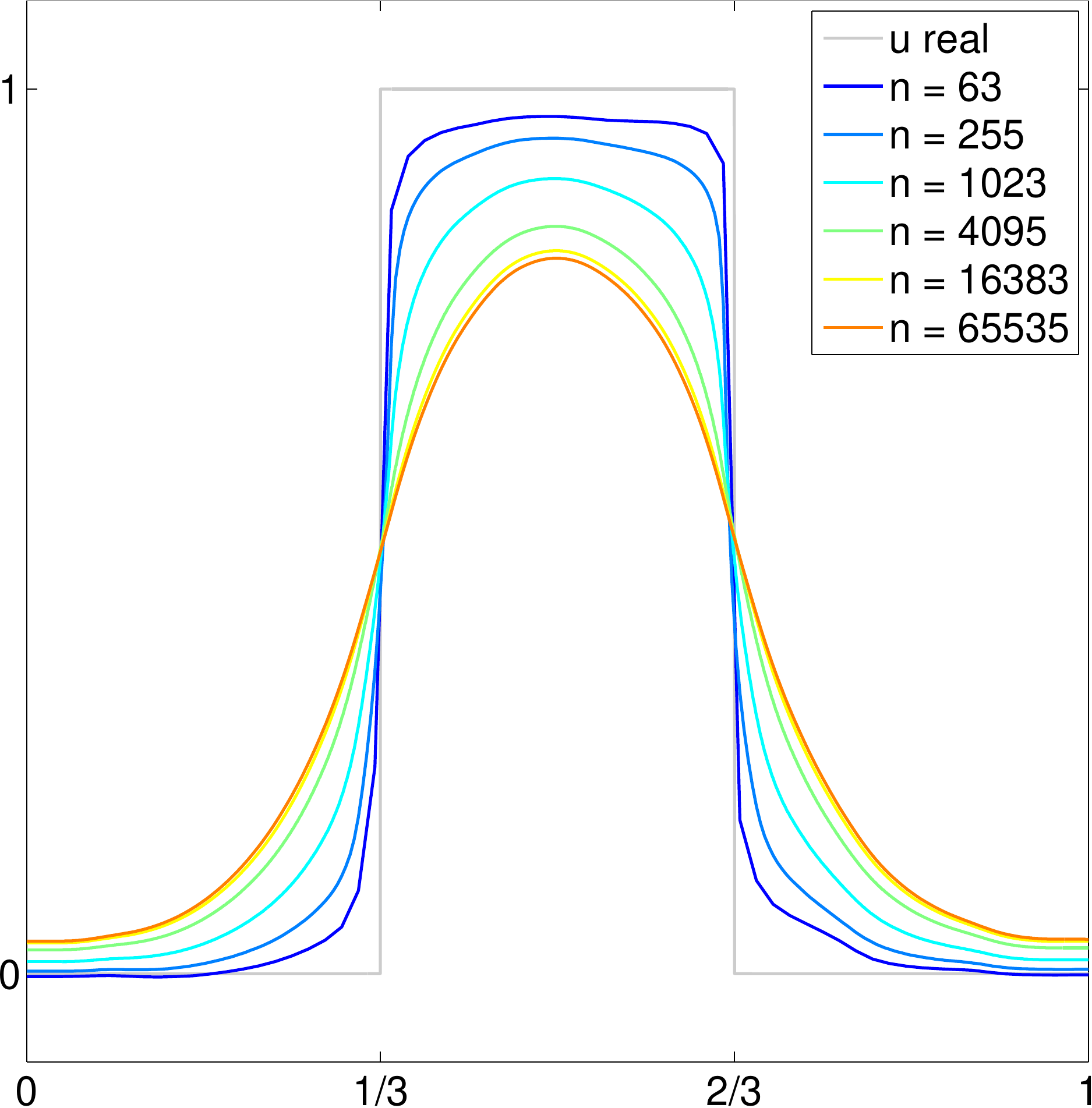}}
\hspace{0.01\textwidth}
\subfloat[][MAP]{\includegraphics[height = 0.49 \textwidth]{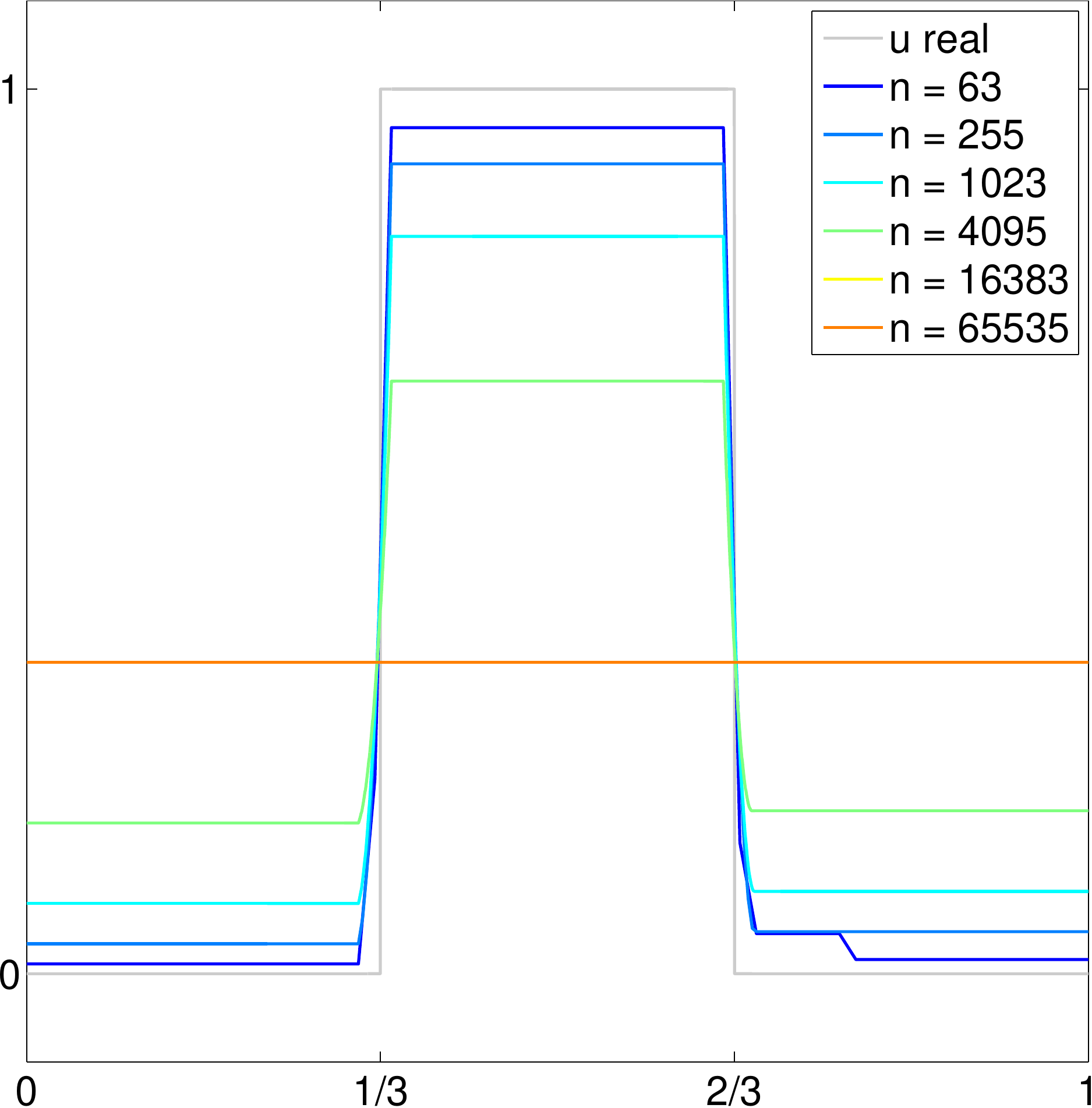}}
\caption{CM and MAP estimates (colored lines) for $\lambda_n \propto  \sqrt{n+1}$ and increasing values of $n$ vs. real solution $\tilde{u}(t)$ (gray line) \label{fig:TVsqrt}}
\end{figure}

\begin{figure}[tb]
   \centering
\subfloat[][CM]{\includegraphics[height = 0.49 \textwidth]{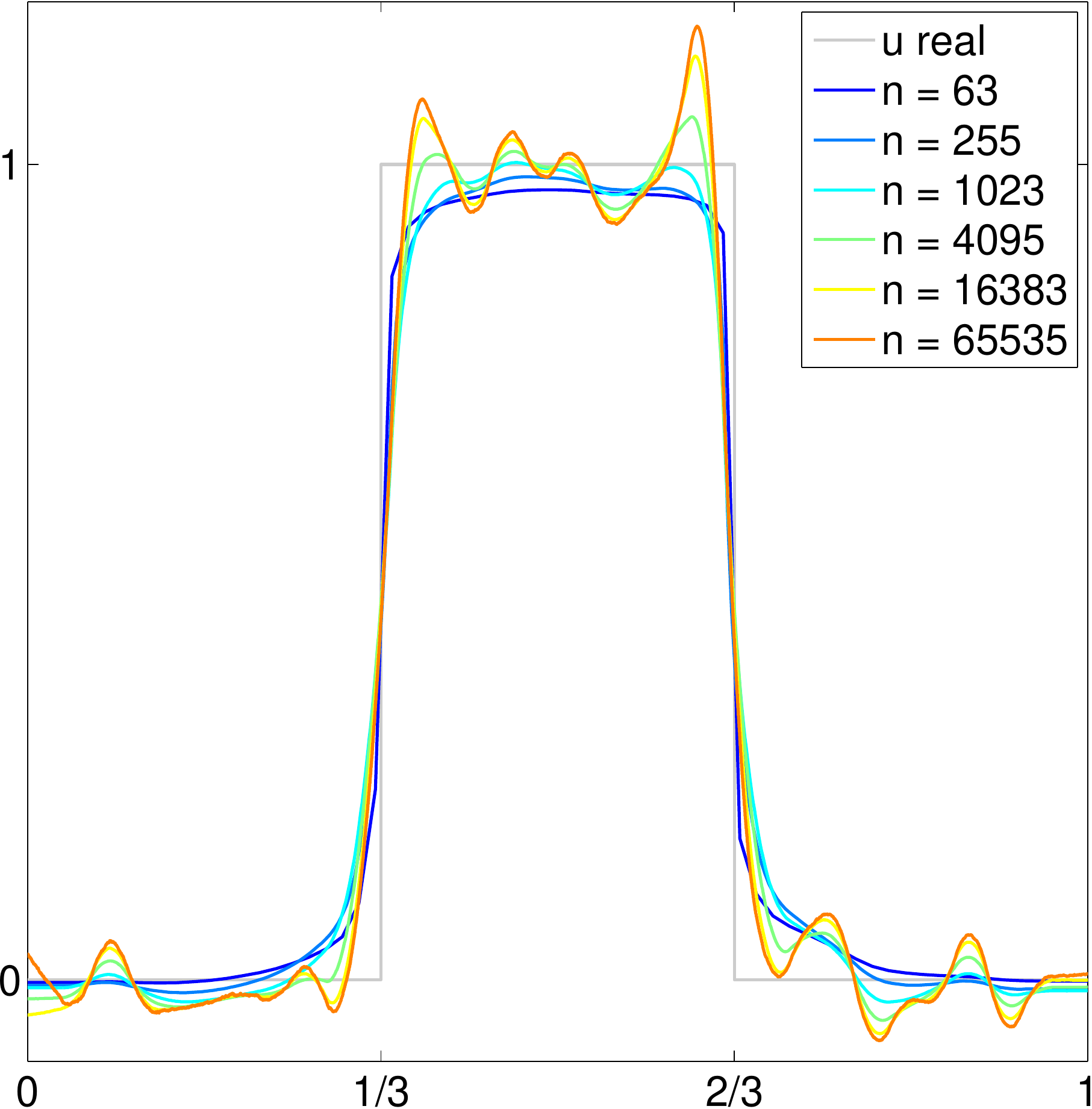}}
\hspace{0.01\textwidth}
\subfloat[][MAP]{\includegraphics[height = 0.49 \textwidth]{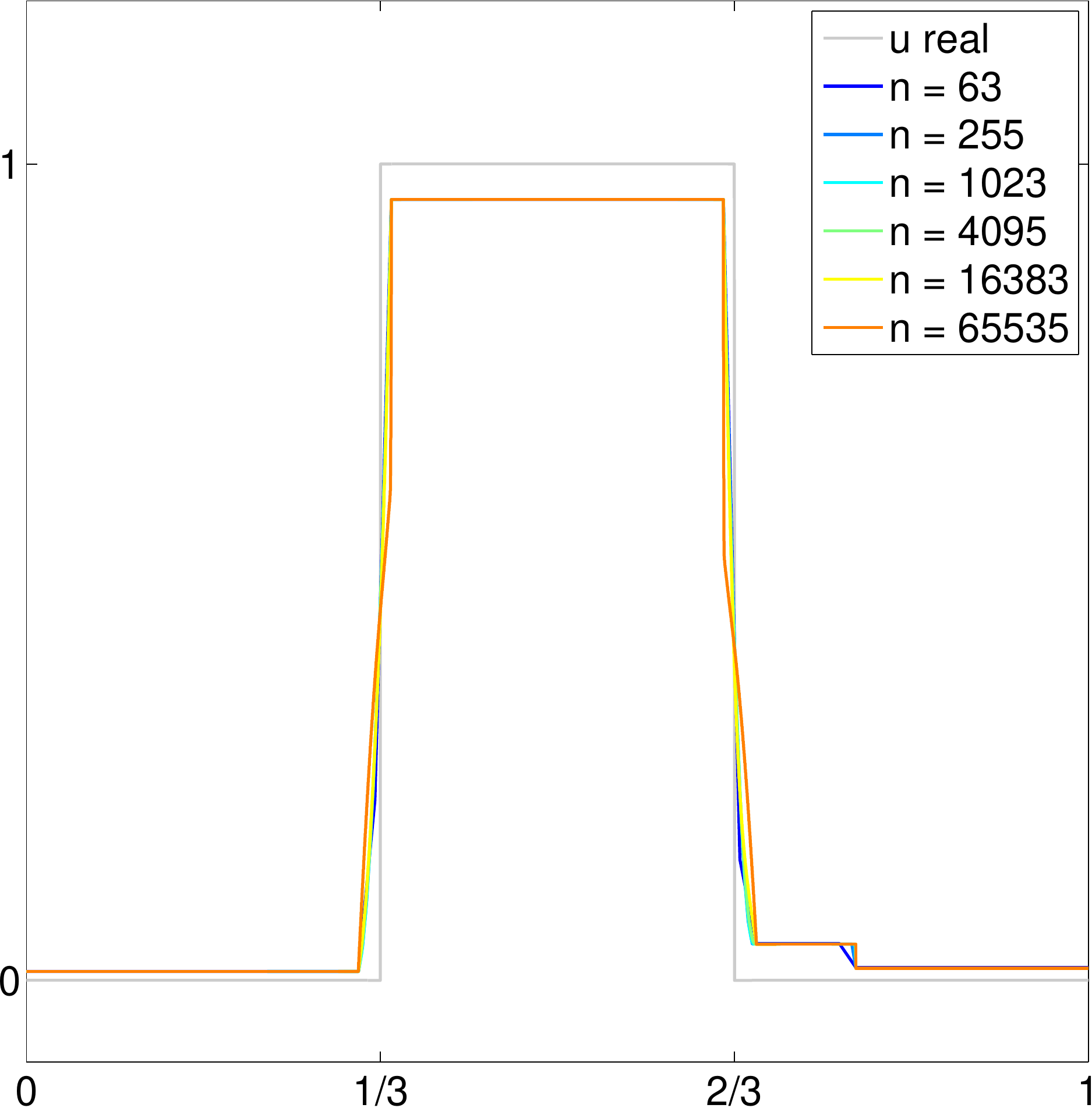}}
\caption{CM and MAP estimates  (colored lines) for $\lambda_n = const.$ and increasing values of $n$ vs. real solution $\tilde{u}(t)$ (gray line) \label{fig:TVconst}}
\end{figure}

\begin{figure}[hbt]
   \centering
\subfloat[][Zoom into CM estimates in Figure \ref{fig:TVconst} \label{subfig:TVconstZoom}]{\includegraphics[height = 0.49 \textwidth]{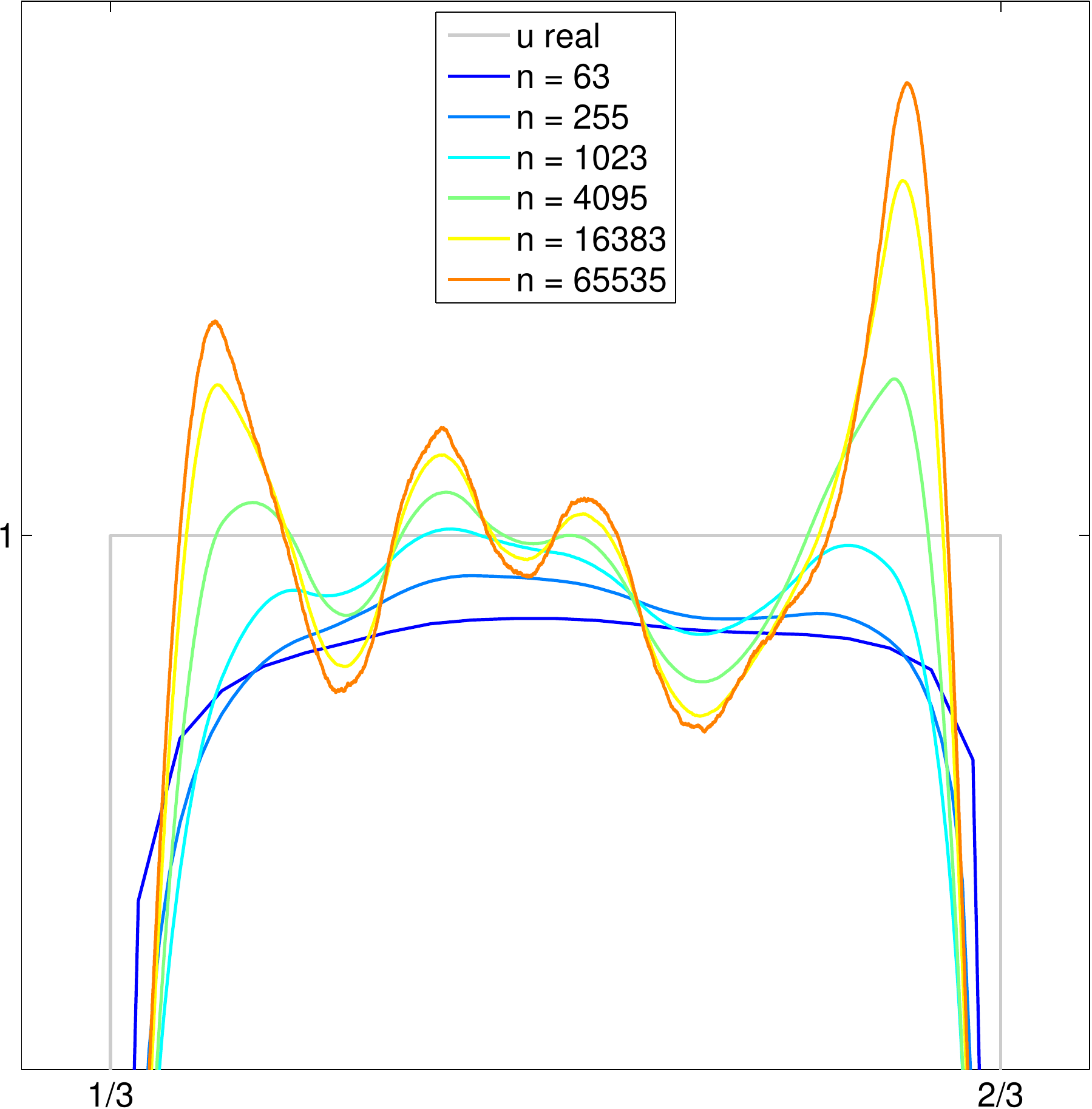}}
\hspace{0.01\textwidth}
\subfloat[][MCMC convergence check \label{subfig:TVconstCheck}]{\includegraphics[height = 0.49 \textwidth]{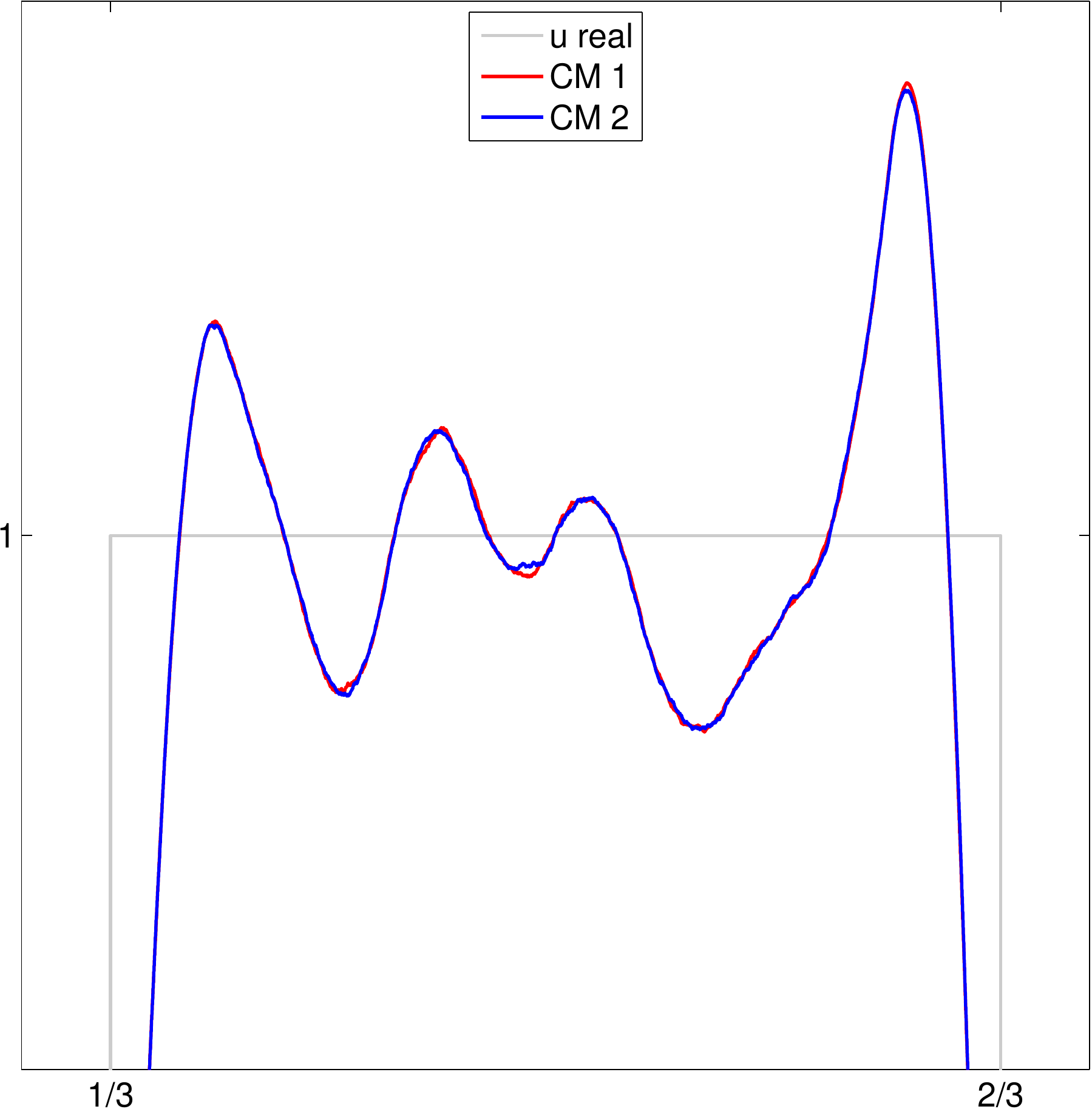}}
\caption{Details of CM estimates in Figure \ref{fig:TVconst}.}
\end{figure}

\subsubsection{Limited Angle CT with Besov Priors}
The discretization dilemma of the TV prior led to a search for edge-preserving and discretization invariant priors. In \cite{LaSaSi09,KoLaNiSi12,HmKaKoLaNiSi13} \emph{Besov space priors}, which rely on a weighted $\ell_1$ norm of wavelet basis coefficients were shown to have these properties. Similar to \cite{HmKaKoLaNiSi13}, we will use Besov priors to reconstruct the Shepp-Logan phantom image $\tilde{u}$ (see Figure \ref{subfig:CTreal}) from the integration of its Radon-projections  into $500$ noisy measurement pixel ($\sigma = 0.01 \cdot \norm{K\tilde{u}}_\infty$) using only $45$ projection angles (see Figure \ref{subfig:CTdata}). In Figure \ref{fig:CTcomp}, CM and MAP estimates for increasing $n$ are shown. As in \cite{KoLaNiSi12}, a Haar wavelet base was used to implement the Besov prior and $\lambda$ was chosen by the \emph{S-curve} method, i.e., depending on the sparsity of the true solution. While a closer examination of such scenarios with real data is subject to future research, for this article, the comparison between CM and MAP estimates is most important: In line with \cite{LaSaSi09,KoLaNiSi12,HmKaKoLaNiSi13},
 we can confirm that both estimates converge in the limit of $n \rightarrow \infty$ and, more importantly, that they are nearly identical. 
\begin{figure}[h]
   \centering
\subfloat[][Unknown function $\tilde{u}$ \label{subfig:CTreal}]{\includegraphics[height = 0.45 \textwidth]{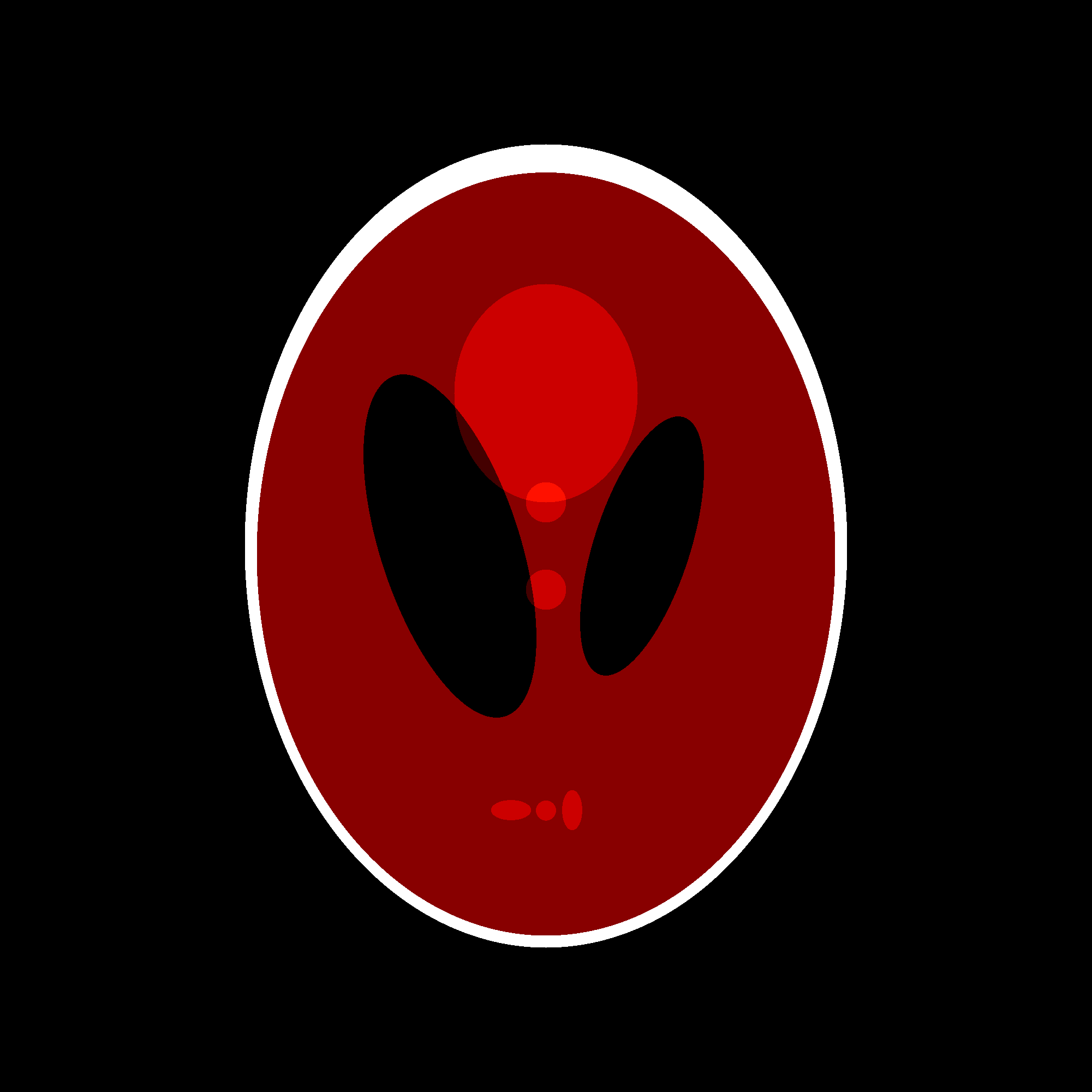}}
\hspace{0.02\textwidth}
\subfloat[][Data $f$ \label{subfig:CTdata}]{
\hspace{0.1\textwidth}
\includegraphics[height = 0.45 \textwidth]{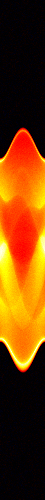}
\hspace{0.1\textwidth}
}
\subfloat[][Color scale]{
\hspace{0.05\textwidth}
\fbox{\includegraphics[height = 0.45\textwidth]{cool2hot.png}}
\hspace{0.05\textwidth}
}
\caption{A 2D limited angle computerized tomography imaging scenario. \label{fig:CTscenario}}
\end{figure}

\begin{figure}[p]
\centering
\subfloat{\vbox to 0.3 \textwidth {\vfil \hbox to 0.2 \textwidth{$n = 64 \times 64$} \vfil }}
\hspace{0.02\textwidth}
\subfloat{\includegraphics[height = 0.3 \textwidth]{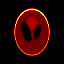}}
\hspace{0.02\textwidth}
\subfloat{\includegraphics[height = 0.3 \textwidth]{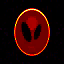}}\\
\subfloat{\vbox to 0.3 \textwidth {\vfil \hbox to 0.2 \textwidth{$n = 128 \times 128$} \vfil }}
\hspace{0.02\textwidth}
\subfloat{\includegraphics[height = 0.3 \textwidth]{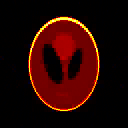}}
\hspace{0.02\textwidth}
\subfloat{\includegraphics[height = 0.3 \textwidth]{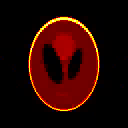}}\\
\subfloat{\vbox to 0.3 \textwidth {\vfil \hbox to 0.2 \textwidth{$n = 256 \times 256$} \vfil }}
\hspace{0.02\textwidth}
\subfloat{\includegraphics[height = 0.3 \textwidth]{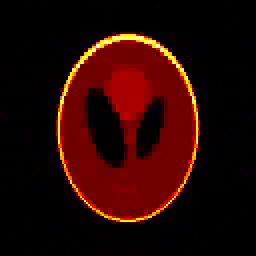}}
\hspace{0.02\textwidth}
\subfloat{\includegraphics[height = 0.3 \textwidth]{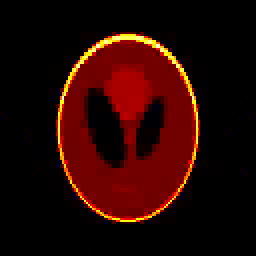}}\\
\subfloat{\vbox to 0.3 \textwidth {\vfil \hbox to 0.2 \textwidth{$n = 512 \times 512$} \vfil }}
\hspace{0.02\textwidth}
\subfloat{\includegraphics[height = 0.3 \textwidth]{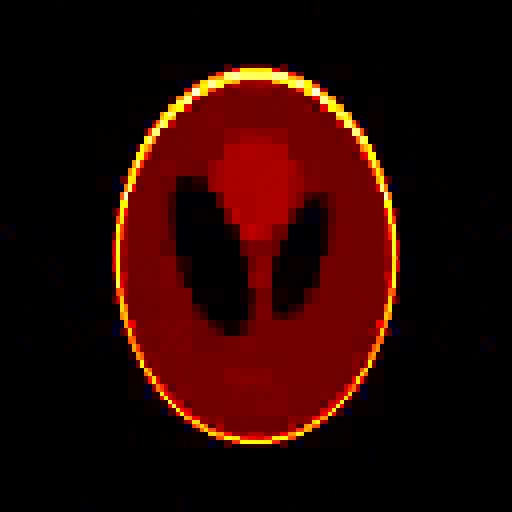}}
\hspace{0.02\textwidth}
\subfloat{\includegraphics[height = 0.3 \textwidth]{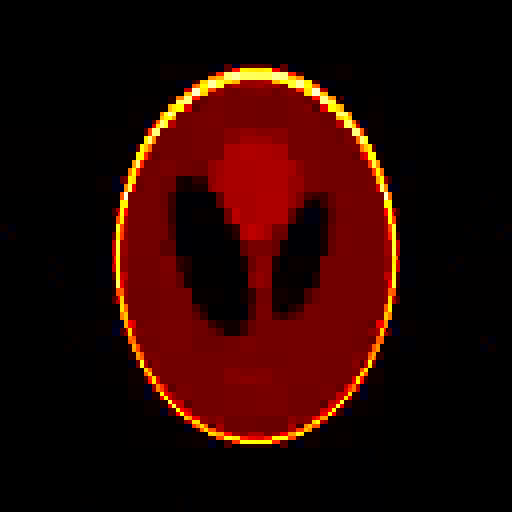}}\\
\subfloat{\vbox to 0.3 \textwidth {\vfil \hbox to 0.2 \textwidth{$n = 1024 \times 1024$} \vfil }}
\hspace{0.02\textwidth}
\subfloat{\includegraphics[height = 0.3 \textwidth]{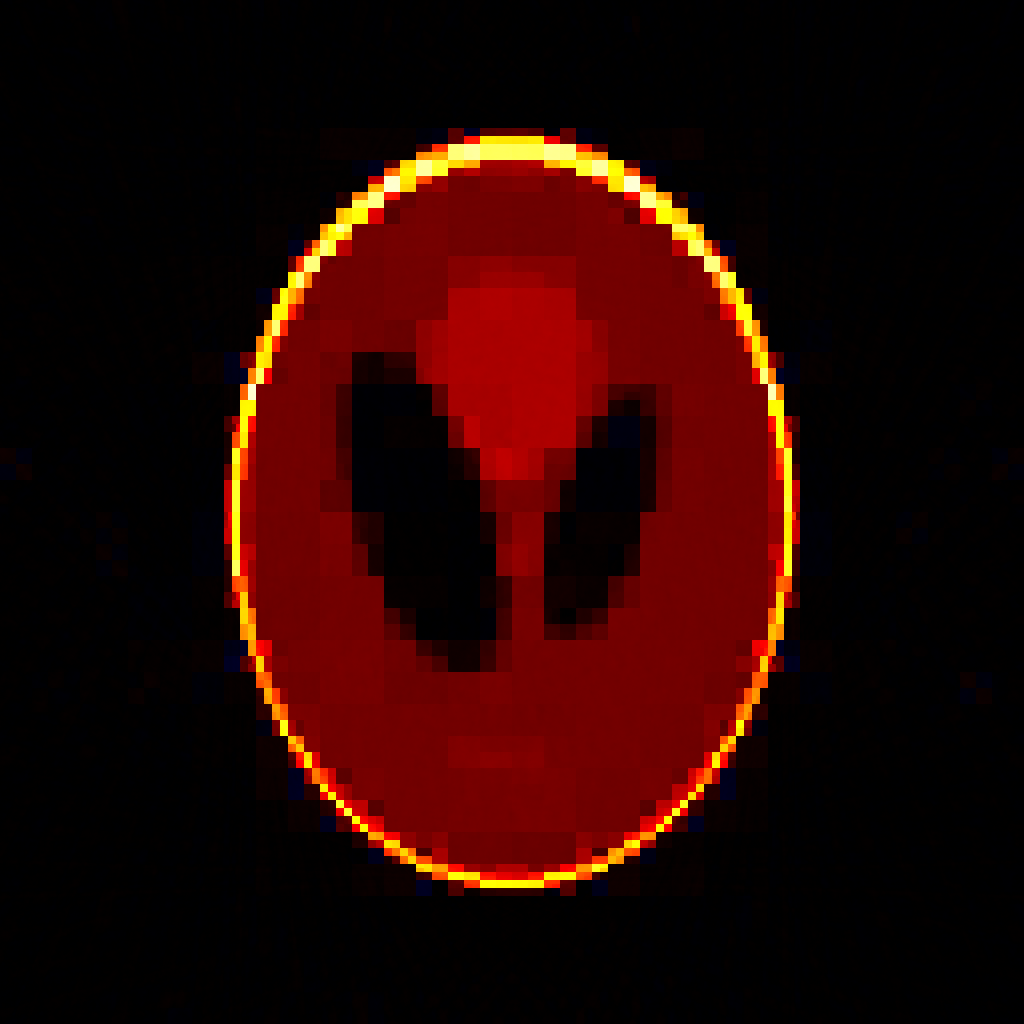}}
\hspace{0.02\textwidth}
\subfloat{\includegraphics[height = 0.3 \textwidth]{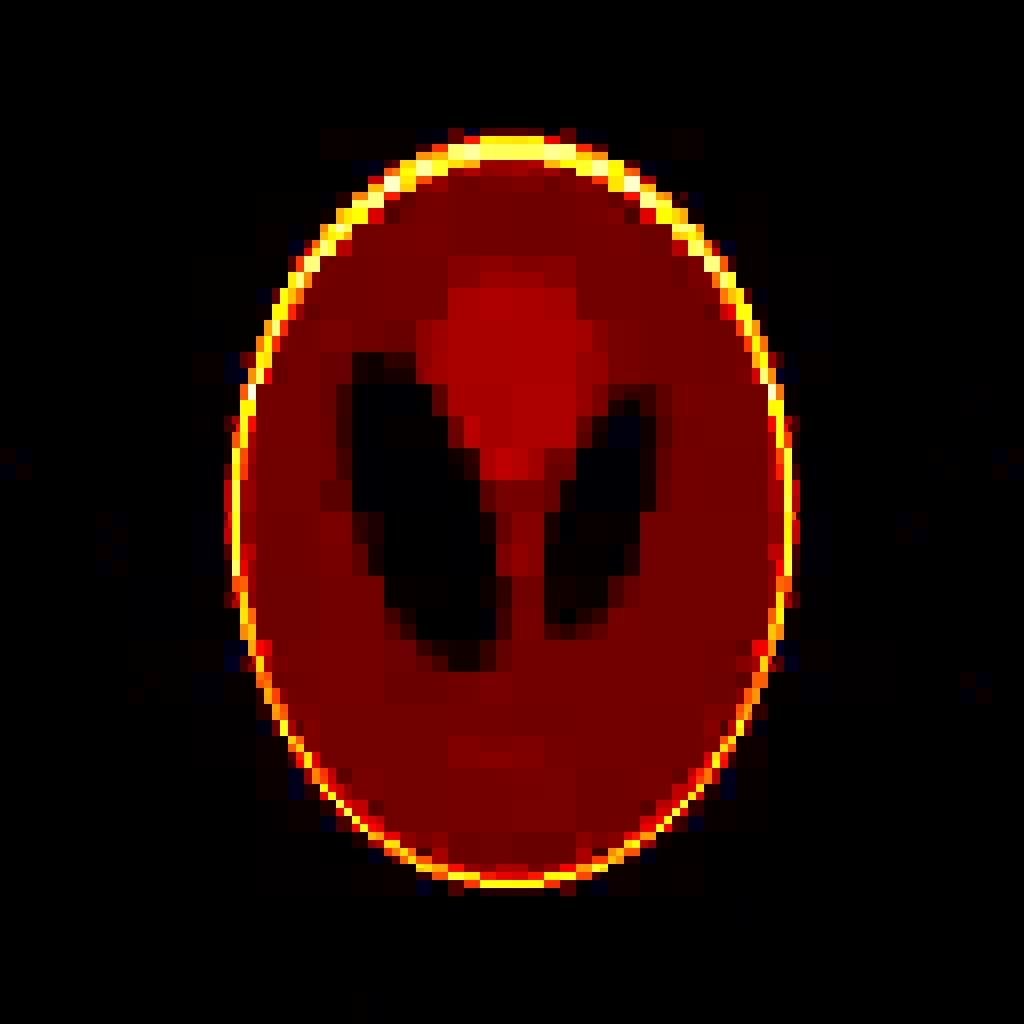}}
\caption{CM (left image column) and MAP (right image column) estimates for increasing resolution of the reconstruction grid.\label{fig:CTcomp}}
\end{figure}

\subsubsection{Summary of Observations and Discussions}
 \begin{itemize}
  \item For Gaussian priors, MAP and CM estimates coincide.
  \item For the first two $\ell_1$-type priors we examined, the MAP estimates were more convincing. For the Besov prior, MAP and CM estimate visually coincide. These findings are similar to a large variety of computational experiments. Loosely speaking, a good CM estimate resembles the MAP estimate.
  \item In \cite{LuPuBuWo12}, we used hierarchical Bayesian modeling (cf. Section~\ref{subsec:SpIP}) for EEG source imaging. While the multimodality of the posterior complicates inference for such priors, suitably computed MAP estimates, again, outperformed the CM estimates. 
  \item Recently, \cite{Gr11,GrPi13,LoLi13} revealed that every CM estimate for a prior $p(u)$ is also a MAP estimate for a different prior $\tilde{p}(u)$. Their intention was to warn against the common "reverse reading" of designing a particular $\J(u)$ for recovering certain classes of real solutions with \eref{eq:GenTikh} and then claiming to perform Bayesian MAP estimation with the prior $p(u) = \exp(-\lambda \J(u))$. In \cite{GrCeDa12}, it was shown that this MAP estimate is usually not very well suited to recover solutions $u$ that are really distributed like  $\exp(-\lambda \J(u))$ (cf. Figure \ref{fig:RndDraws}, none of the true solutions used in the computational scenarios fitted to the assumed priors!). For the discussion in this article, these results mean that a general discrimination of MAP estimates based on the Bayes cost formalism only makes sense if one strongly believes that the chosen prior most accurately models the distribution of the real solution. Otherwise, one ends up in the contradiction that the appraised CM estimate will simultaneously be a discredited MAP estimate (just for another prior). 
 \end{itemize}
The next sections will present new theoretical ideas that resolve the contradictions between these observations and the classical view on the comparison between MAP and CM estimates.
 
\section{A Novel Characterization of the MAP Estimate} \label{sec:RehabMAP}
This section presents a novel Bayes cost approach to MAP estimates for log-concave priors of the form $\exp(-\lambda \J(u))$. Throughout this section we assume that $\J:\R^n \rightarrow \R\cup \{\infty\}$ is a Lipschitz-continuous convex functional, such that for $\lambda > 0$ the function $u \mapsto \Vert Ku \Vert^2 + \lambda \J(u)$ has at least linear growth at infinity. Due to Rademacher's theorem (cf. \cite{EvGa91}) this implies that $p_{po}(u|f)$ is log-concave and differentiable almost everywhere in $\R^n$. The main ingredient will be the (generalized) \emph{Bregman distance}:
\begin{defi}
For a convex functional $\J:\R^n \longrightarrow \R\cup \{\infty\}$, the Bregman distance $D_\J^q(u,v)$ between $u, v \in \R^n$ for a subgradient $q \in \partial \J(v)$ is defined as 
\begin{equation}
	D_\J^q(u,v) = \J(u) - \J(v) - \langle q, u - v\rangle, \qquad q \in \partial \J(v) \label{eq:BregDef}
\end{equation}
\end{defi}
Note that if $\J(u)$ is Fr\'echet-differentiable in $v$, $q$ is the Fr\'echet derivative $\J'(v)$. We will therefore simplify the notation to $D_\J(u,v)$, and use $D_\J^q(u,v)$ only if we want to stress the potential ambiguity. Table \ref{tbl:BregDist} lists the Bregman distances induced by some Gibbs energies $\J(u)$. Figure \ref{fig:Breg} gives an illustration: Basically, $D_\J(u,v)$ measures the difference between $\J$ and its linearization in $u$ at another point $v$. Further, $D_\J(u,v) \geqslant 0$ and for strictly convex $\J(u)$, $D_\J(u,v) = 0$ implies $u=v$. However, the Bregman distance is not a distance in the usual mathematical sense, i.e., a metric, as it is, in general, neither symmetric nor satisfies the triangle inequality. We will further use that $D_\J(u,v)$ is convex in $u$. Bregman distances have become an important tool in variational regularization, e.g., to derive error estimates and convergence rates \cite{BuOs04,Be11}, to enhance inverse methods by Bregman iterations \cite{BuReHe07,Mo12} or to develop optimization schemes like the Split-Bregman algorithm \cite{GoOs09} used in this paper. 
\begin{figure}[bt]
   \centering
\subfloat[][$\J(x) = x^2$ \label{subfig:BregDistL2}]{\includegraphics[height = 0.4 \textwidth]{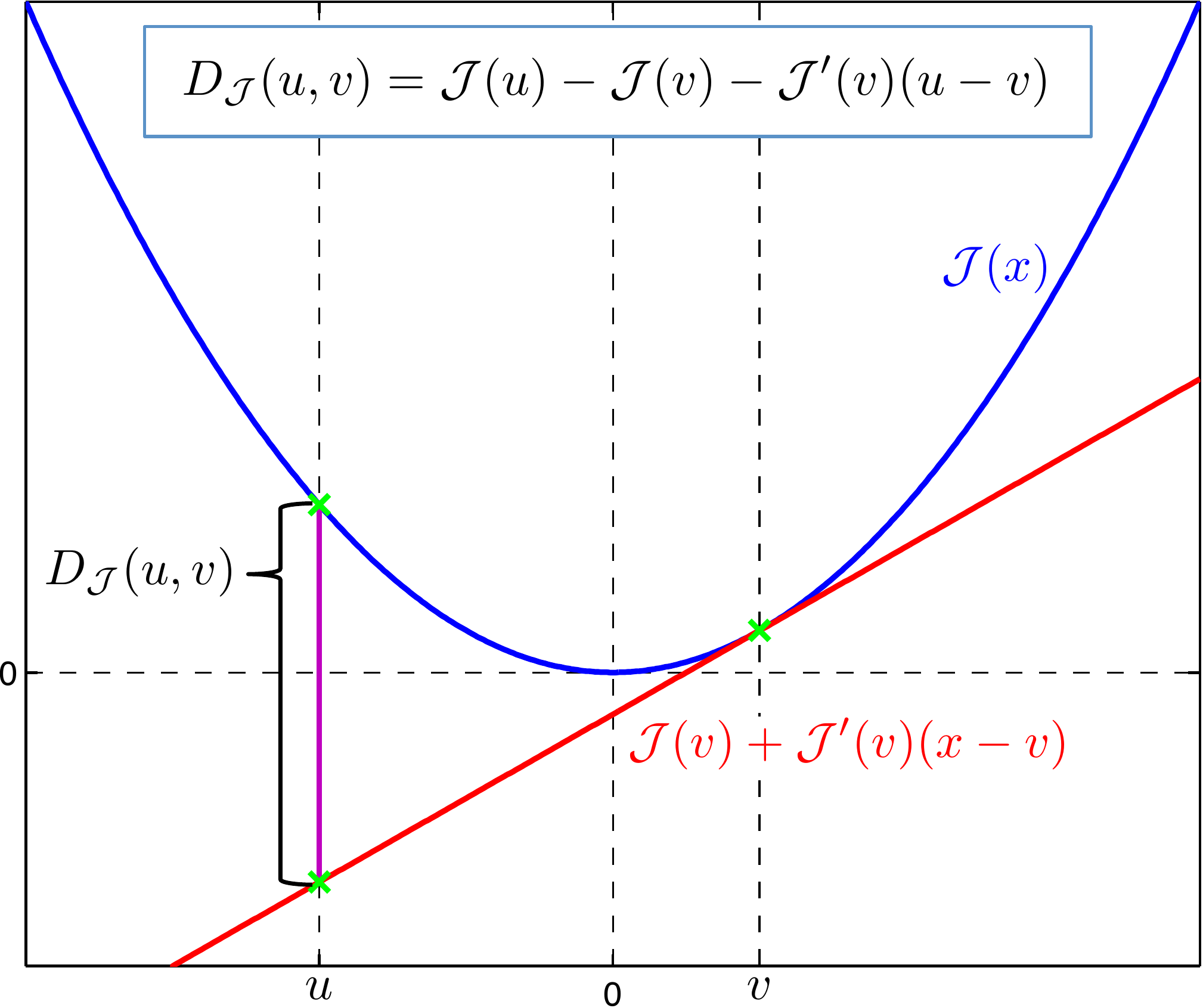}}
\hspace{0.01\textwidth}
\subfloat[][$\J(x) = |x|$ \label{subfig:BregDistL1}]{\includegraphics[height = 0.4 \textwidth]{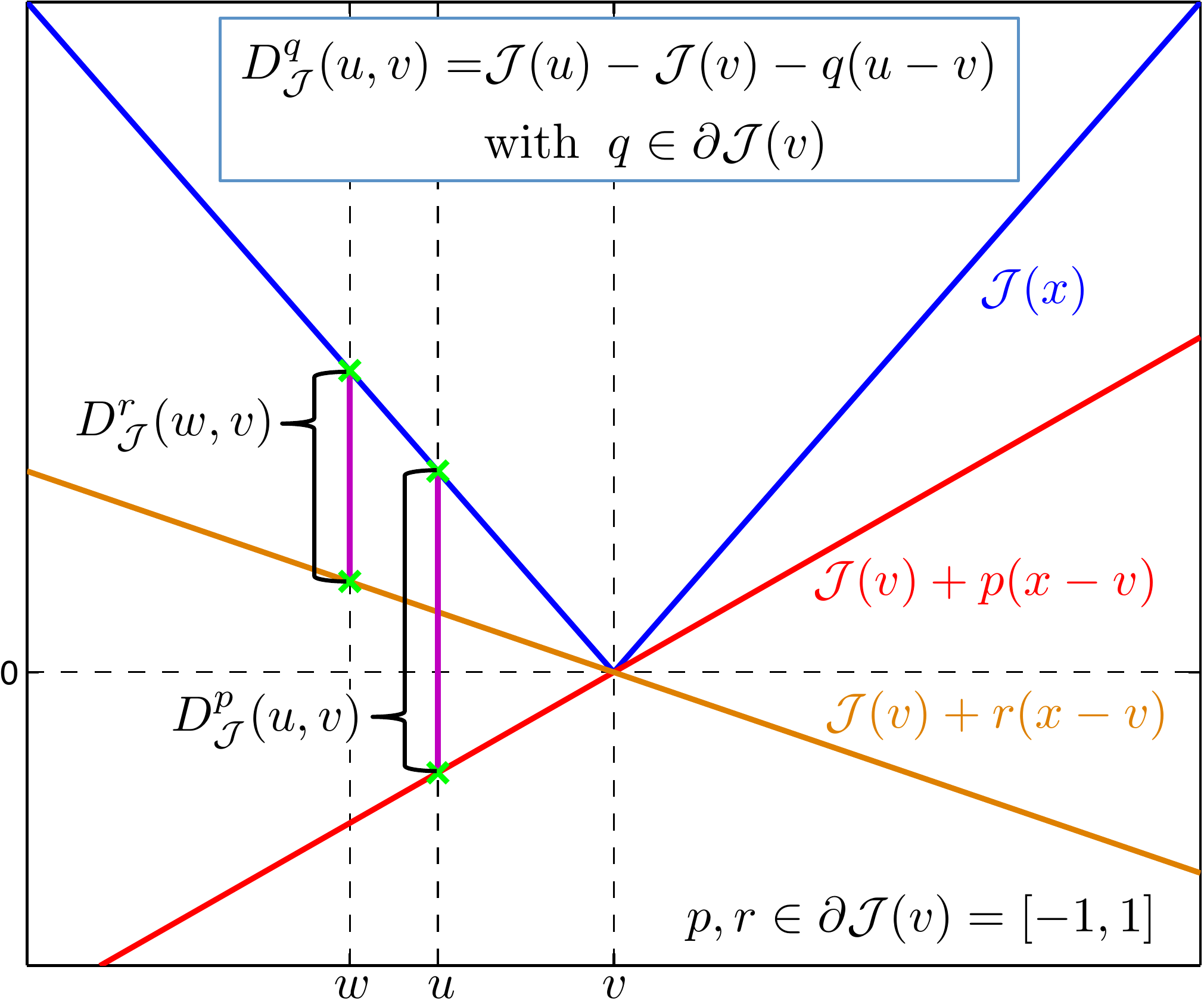}}
\caption{Illustrative explanation of the Bregman distance.\label{fig:Breg}}
\end{figure}
\begin{table}
\caption{\label{tbl:BregDist} Bregman distances induced by some Gibbs energies $\J(u)$ commonly used for prior modeling. Note that if $\J$ is separable, so is $D_\J^q(u,v)$. In these cases, the scalar expression are listed, only.} 
\begin{indented}
\lineup
\item[]\begin{tabular}{lll}
\br                              
 $\J(u)$&  dom$(\J)$&$D_\J(u,v)$ \cr 
\mr
$\frac{1}{2} \sqnorm{L u} $& $\R^n$ & $\frac{1}{2}  \sqnorm{L (u-v)}$  \cr
$|u|^p$,  $(1 < p < \infty)$  &  $\R$ &  $|u|^p - p \, u \, \sign(v) |v|^{p-1} + (p-1) |v|^p$ \cr
$|u| $&  $\R$ & $ \left( \sign(u) - \sign(v) \right) u $ \cr
$ u \log u-u$ &  $\R_{\geqslant 0}$& $ u \log \frac{u}{v}+v-u$ (Kullback-Leibler divergence)\cr
\br
\end{tabular}
\end{indented}
\end{table}

\subsection{New Bayes Cost Functions} \label{subsec:NewCost}
The classical discrimination of the MAP estimate as only being asymptotically a Bayes estimator for the uniform cost \eref{eq:UniCost} (cf. Section~\ref{subsec:BayesCost}) has a crucial flaw: It does not mean that the MAP estimate cannot be a proper Bayes estimator for a different cost function. This suggests that one should search for alternative costs better suited to the asymptotic Banach space structure such as Bregman distance costs:

\begin{defi} \label{defi:CostFun}
Let $L \in \R^{n \times n}$ be regular and $\beta > 0$. Define
\begin{eqnarray}
\PsiLS   &:= \linorm{ K (\hat{u} - u)} + \beta \sqnorm{L (\hat{u} - u)}\\
\PsiBrg &:= \linorm{ K (\hat{u} - u)} + 2 \lambda D_\J(\hat{u},u)
\end{eqnarray}
\end{defi}
Both $\PsiLS$ and $\PsiBrg$ are proper, convex (with respect to $\hat{u}$) cost functions. In the following, we will need the decay property 
\begin{equation}
\lim_{R \rightarrow \infty} \int_{\partial \mathcal{B}_R(0)} p_{po}(u|f) \: \rmd u = 0 \label{eq:DecProp}
\end{equation}
which is fulfilled under the linear growth assumption above, which yields for constants $A,B$ independent of $R$
$$ p_{po}(u|f) \leq Ae^{-\frac{B}R} \qquad \mbox{on }~ \mathcal{B}_R(0).$$
%

\begin{theo} \label{theo:Estimators}
Let $\J$ be as above and let $\lambda >0$ and $\beta \geqslant 0$.
Then the CM estimate is a Bayes estimator for $\PsiLS$ and the MAP estimate is a Bayes estimator for $\PsiBrg$. 
\end{theo}
\begin{proof}
We start from \eref{eq:BayesEstimator} and insert the definition of $\PsiLS$:
\begin{equation*}
\fl \hat{u}_{\PsiLSnoarg}(f) = \argminsub{\hat{u}} \left\lbrace \int \left( \linorm{ K (\hat{u} - u)} + \beta \sqnorm{L (\hat{u} - u)} \right) \: p_{po}(u|f) \: \rmd u \right\rbrace
\end{equation*}
We can rewrite the above by inserting $\uCM$ and expanding squares
\begin{eqnarray*}
\fl \hat{u}_{\PsiLSnoarg}(f) = \argminsub{\hat{u}} \bigg\lbrace  \int \left( \linorm{ K (\hat{u} - \uCM)} + \beta \sqnorm{L (\hat{u} - \uCM)} \right) \: p_{po}(u|f) \: \rmd u  \\
\fl \qquad +  \int \left(  \linorm{ K (u - \uCM)}  + \beta \sqnorm{L (u - \uCM)} \right) \: p_{po}(u|f) \: \rmd u \\
\fl \qquad - 2 \int \left( \langle K (\hat{u} - \uCM), K(u-\uCM)\rangle_{\NoiSig} + \beta \langle L (\hat{u} - \uCM), L(u-\uCM)\rangle_{2}  \right) \: p_{po}(u|f) \: \rmd u \bigg\rbrace 
\end{eqnarray*}
Due to the linearity and the definition of the CM estimate \eref{eq:CMDef}  the last integral vanishes and hence, $\hat u = \uCM$ is obviously a minimizer.
For the MAP estimate, we again start from \eref{eq:BayesEstimator} 
and insert the definition of $\PsiBrg$:
\begin{equation*}
\fl \hat{u}_{\PsiBrgnoarg}(f) = \argminsub{\hat{u}} \left\lbrace \int_{\R^ n} \left( \linorm{ K (\hat{u} - u)} + 2 \lambda D_\J(\hat{u},u) \right) \: p_{po}(u|f) \: \rmd u \right\rbrace
\end{equation*}
Now, we can exclude the null-set where $\J(u)$ is not Fr\'echet-differentiable,
$$\mathcal{S} := \{u \in \R^n | |\partial \J(u)| \neq 1 \}, $$
from the integration and insert the definition of $D_\J(\hat{u},u)$ on $\mathcal{S}^c$:  
\begin{equation*}
\fl \hat{u}_{\PsiBrgnoarg}(f) = \argminsub{\hat{u}} \left\lbrace \int_{ \mathcal{S}^c} \left( \linorm{ K (\hat{u} - u)} + 2 \lambda \left( \J(\hat{u})-\J(u) - \langle \J'(u), \hat{u}-u\rangle \right) \right) p_{po}(u|f) \rmd u \right\rbrace
\end{equation*}
The squared norm can be developed as in the case of the CM-estimate, while for the Bregman distance we use the following elementary identity:
$$ D_\J(\hat{u},u) = D_\J(\hat{u},\uMAP) + D_\J(\uMAP,u) + \langle \pMAP - \J'(u), \hat{u}-\uMAP \rangle  $$
Thus, on $\mathcal{S}^c$ we have
\begin{eqnarray*}
\fl \linorm{K (\hat{u} - u)} + 2 \lambda D_\J(\hat{u},u)  \\
\fl \qquad = \linorm{K(\hat{u} - \uMAP)}+ 2\lambda D_\J(\hat{u},\uMAP) + \linorm{K(\uMAP-u)}+  2\lambda D_\J(\uMAP,u)  \\
\fl \qquad \quad + 2 \langle K (\hat{u} - \uMAP), K(\uMAP-u)\rangle_{\NoiSig} + 2 \lambda \langle \pMAP - \J'(u), \hat{u}-\uMAP \rangle. 
\end{eqnarray*}
The first two terms in the second line are obviously minimal for $\hat{u}=\uMAP$, while the other terms in this line are independent of $\hat{u}$. 
In the last line we can insert the subgradient from the optimality condition for $\uMAP$, 
$$\pMAP = - \frac{1}\lambda K^* \NoiSig (K \uMAP -f)   \in \partial \J(\uMAP),$$
 and rewrite
\begin{eqnarray*}
\fl 2 \langle K (\hat{u} - \uMAP), K(\uMAP-u)\rangle_{\NoiSig} + 2 \lambda \langle \pMAP - \J'(u), \hat{u}-\uMAP \rangle \\ 
  \qquad \quad =  - 2 \langle K (\hat{u} - \uMAP), Ku - f\rangle_{\NoiSig} - 2 \lambda \langle - \J'(u), \hat{u}-\uMAP \rangle \\ 
  \qquad \quad  = - 2 \langle K^* \NoiSig (Ku -f) + \lambda \J'(u), \hat{u}-\uMAP \rangle \\ 
  \qquad \quad = 2 \langle \nabla_u \log p_{po}(u|f) , \hat{u} - \uMAP \rangle.
 \end{eqnarray*}
Using the logarithmic derivative $\nabla_u p_{po}(u|f) = (\nabla_u \log p_{po}(u|f)) p_{po}(u|f)$, the posterior expectation of the latter equals
\begin{eqnarray*}
\fl 2 \int_{ \mathcal{S}^c} \langle \nabla_u \log p_{po}(u|f) , \hat{u} - \uMAP \rangle \: p_{po}(u|f) \: \rmd u &=& 
2 \langle \int_{ \mathcal{S}^c}  \nabla_u p_{po}(u|f) \: \rmd u, \hat{u} - \uMAP \rangle.
 \end{eqnarray*}
With Gauss' theorem and \eref{eq:DecProp} we finally obtain:
\begin{eqnarray*}
\left\Vert \int \nabla_u p_{po}(u|f) \: \rmd u \right\Vert \quad \: &= \quad &\lim_{R \rightarrow \infty} \left\Vert \int_{\mathcal{B}_R(0)} \nabla_u p_{po}(u|f) \: \rmd u \right\Vert \nonumber \\
&= &\lim_{R \rightarrow \infty} \left\Vert \int_{\partial \mathcal{B}_R(0)} p_{po}(u|f) \frac{u}{R} \: \rmd u \right\Vert \nonumber \\
&\leqslant &\lim_{R \rightarrow \infty}\int_{\partial \mathcal{B}_R(0)} p_{po}(u|f) \: \rmd u  \nonumber \\
 &= \quad & \quad 0 
\end{eqnarray*}
\end{proof}

First, we apply Theorem \ref{theo:Estimators} to the fundamental case of Gaussian priors. We can parameterize any (centered) Gaussian energy as $\J(u) = \beta/(2\lambda) \sqnorm{ L  u }$. For this choice $2\lambda D_\J(\hat{u},u)$ = $\beta \sqnorm{L (\hat{u} - u)}$, and $\Psi_{\text{\tiny LS}}(u,\hat{u})  = \Psi_{\text{\tiny Brg}}(u,\hat{u})$: The equality of MAP and CM estimate in the Gaussian case is no longer a strange coincidence but follows naturally from the properties of the Bregman distance. \\
In the non-Gaussian case, the domain of $\J$ usually defines a Banach space or a subset thereof in the limit $n\rightarrow \infty$. E.g., the discrete total variation prior will define the space of functions of bounded variation in the limit. In such a space there is no natural Hilbert space norm that one should obtain as the limit of $\Vert L u \Vert^2$. Even worse, it is questionable whether any Hilbert space norm is a meaningful measure for functions of bounded variation. The only reasonable choice might be $L = 0$, which means that $\PsiLS$ measures purely in the output space, which will be a Hilbert space. However, for ill-posed inverse problems with noisy data it is well-established that one should not just minimize a criterion related to the output $Ku$.

\subsection{A MAP-Centered Form of the Posterior} \label{subsec:Center}

As pointed out in Section~\ref{subsec:BayesCost}, one classical geometrical argument was that the CM estimate is in the center of mass of $p_{po}(u|f)$ while the MAP estimate does not allow for such an interpretation (cf. Section~\ref{subsec:BayesCost}). Using Bregman distances, we can rewrite the $p_{po}(u|f)$ in a MAP-centered form, which also disqualifies this argument. We use the optimality condition of the MAP-estimate \eref{eq:GenTikh},
 \begin{equation}
   K^* \NoiSig (K \uMAP - f) + \lambda \pMAP = 0, \qquad \pMAP \in \partial \J(\uMAP), \label{eq:OptMAP}
 \end{equation}
to rewrite $K^* \NoiSig f$ in the posterior energy:
\begin{eqnarray}
\fl  \frac{1}{2} \linorm{K \, u - f} + \lambda \J (u) \nonumber \\
= \frac{1}{2} \linorm{K \, u} - \langle K^* \NoiSig f, u \rangle + \lambda \J (u) + \frac{1}{2 } \linorm{f} \nonumber \\
 = \frac{1}{2} \linorm{K \, u} - \langle K^* \NoiSig K \uMAP + \lambda \pMAP , u \rangle + \lambda \J (u) + \frac{1}{2 } \linorm{f} \nonumber \\
 = \frac{1}{2} \linorm{K \, u} - \langle \NoiSig K \uMAP  , K u \rangle + \frac{1}{2 } \linorm{K \uMAP} \nonumber \\
\quad + \lambda \left( \J (u) - \J(\uMAP) - \langle \pMAP,u -\uMAP \rangle \right) \nonumber \\
\quad - \frac{1}{2} \linorm{K \uMAP} + \lambda \left(\J(\uMAP) - \langle \pMAP, \uMAP \rangle \right) + \frac{1}{2 } \linorm{f} \nonumber \\
 = \frac{1}{2} \linorm{K \, (u-\uMAP)}  + \lambda D^{\pMAP}_\J(u,\uMAP) + \text{const.}, \label{eq:PostMAPEnergy}
\end{eqnarray}
where $\text{const.}$ sums all terms not depending on $u$. Hence, we can write the posterior as
\begin{equation}
p_{po}(u|f) \propto \exp \left( - \frac{1}{2} \linorm{K (u- \uMAP)}  - \lambda D_\mathcal{J}^{\pMAP}(u,\uMAP) \right).
\end{equation}
 Now, the posterior energy is sum of two convex functionals both minimized by $\uMAP$, i.e., $\uMAP$ is the center of $p_{po}(u|f)$ with respect to the distance induced by \eref{eq:PostMAPEnergy}.

\subsection{Average Optimality of the CM Estimate}
To further compare MAP and CM estimates, we derive an ``average optimality condition'' for the CM estimate. Let
\begin{equation}
\pCM := \Exp\left[ \J'(u) \right] = \int \J'(u) p_{po}(u|f) \rmd u
\end{equation}
be the CM estimate for the (sub)gradient of $\J(u)$. We have:
\begin{eqnarray}
\fl \qquad   K^* \NoiSig (K \uCM -f) + \lambda \pCM  &=& K^* (K \NoiSig \Exp[u] -f) + \lambda \Exp[\mathcal{J}'(u)] \nonumber \\
&=& \Exp \left[ K^* \NoiSig (K u -f) + \lambda \mathcal{J}'(u) \right] \nonumber \\
&=&  \int_{\mathcal{S}^c} K^* \NoiSig (K u -f) + \lambda \mathcal{J}'(u) \: p_{po}(u|f) \: \rmd u \nonumber \\
&=&  \int_{\mathcal{S}^c}  \nabla_u p_{po}(u|f) \: \rmd u = 0, \label{eq:OptCM}
\end{eqnarray}
where the integral term, again, vanishes. Comparing \eref{eq:OptCM} to \eref{eq:OptMAP} we see that the CM estimate fulfills an optimality condition "on average", i.e., with respect to the average gradient $\pCM =  \Exp[\J'(u)]$ but not with respect to the gradient $\J'(\uCM) = \J'(\Exp[u])$. The difference between MAP and CM estimate here manifests in $\J'(\Exp[u]) \neq \Exp[\J'(u)]$, which, again, vanishes for the Gaussian case where $\J'(u)$ is linear.

\subsection{New Inequalities } \label{subsec:InEq}
 Finally, we show that when measured in the Bregman distance $D_J(\hat{u},u)$, which is a more reasonable error measure than norms in the case of a non-quadratic $\J(u)$, the MAP estimate performs better than the CM estimate. In return, the CM estimate out-performs the MAP estimate when the error is measured in a quadratic distance:
 \begin{theo} \label{theo:InEq}
 Let $L \in \R^{n \times n}$ be regular, then we have
  \begin{eqnarray}
  \Exp \left[ \sqnorm{L (\uCM - u )} \right]  &\leqslant \Exp \left[ \sqnorm{ L (\uMAP - u ) } \right]  \\
  \Exp  \left[ D_\mathcal{J}(\uMAP,u) \right]   &\leqslant \Exp  \left[ D_\J(\uCM,u) \right]
  \end{eqnarray}
 \end{theo}
\begin{proof}
 The first inequality directly follows from the fact that $\uCM$ is also the Bayes estimator for $\Psi(u,\hat{u}) = \sqnorm{L (\uCM - u )}$, which follows from the proof to Theorem~\ref{theo:Estimators}. For the second inequality, we use the minimizing properties of MAP and CM estimates:
\begin{eqnarray*}
\fl \int \left( \linorm{ K (\uMAP - u)} + 2 \lambda D_\J(\uMAP,u) \right) \: p_{po}(u|f) \: \rmd u \\
 \leqslant \int \left( \linorm{ K (\uCM - u)} + 2 \lambda D_\J(\uCM,u) \right) \: p_{po}(u|f) \: \rmd u\\
 \leqslant \int \left( \linorm{ K (\uMAP - u)} + 2 \lambda D_\J(\uCM,u) \right)\: p_{po}(u|f) \: \rmd u \\*
 \quad + \beta \int \left(  \sqnorm{L (\uMAP - u)} - \sqnorm{L (\uCM - u)} \right) \: p_{po}(u|f)  \: \rmd u\\
\end{eqnarray*}
Since $\beta > 0$ is arbitrary, we can consider $\beta \rightarrow 0$ and obtain
\begin{equation*}
\int D_\J(\uMAP,u) \: p_{po}(u|f)  \: \rmd u \leqslant \int D_\J(\uCM,u) \: p_{po}(u|f)  \: \rmd u
\end{equation*}
\end{proof}


\section{Discussion and Conclusions}\label{sec:Discussion}

In this article, we examined point estimates in Bayesian inversion from both computational and theoretical perspectives and contrasted recent observations with classical assumptions. We showed that the common discrimination of MAP estimates based on the Bayes cost formalism is not valid: Using Bregman distances, the MAP estimate is a proper Bayes estimator for a convex cost function as well (Section~\ref{subsec:NewCost}). Further aspects like the centeredness of the posterior around the MAP estimate (Section~\ref{subsec:Center}) and the optimality with respect to error measures (Section~\ref{subsec:InEq}) were examined as well.\\ 
A potential irritation might be that the cost function for the MAP estimate depends on the chosen prior while the one for the CM estimate does not. However, this is usually not a drawback but rather an advantage: $\J(u)$ is chosen such that it grasps the most distinctive features of $u$. Often, one is consequently also most interested in estimating these features correctly, which is measured by $D_\J(u,v)$ better than in some squared error metric. For instance, in a situation like the 2D image deblurring scenario in Section~\ref{subsubsec:2DImDeb}, one is mainly interested in the correct separation and location of the intensity spots while their absolute amplitudes might be of minor interest. In such situations, the standard squared error is a poor indicator of reconstruction quality (see also the discussions in \cite{Be11,BuOs04,BuReHe07,ScKaHoKa12}). On the other hand, the induced Bregman distance $D_\J(u,v)$ is $0$ if the sign pattern of $u$ and $v$ coincide (cf. Table \ref{tbl:BregDist}) and grows only linearly, otherwise. \\
The main aim of this article was to rehabilitate the MAP estimate for Bayesian inversion and, thereby, to also disprove common misconceptions about the nature of MAP estimation.
We think ''MAP or CM?'' is \emph{not} an interesting question for relating variational regularization and Bayesian inference. It might be an obvious question but it puts the focus on a direct comparison between point estimates and suggests that one should choose between one of the two approaches. The real strength of Bayesian approaches is to model and quantify uncertainty and information at all stages of the problem, \emph{beyond} point estimates. In this direction, Bayesian techniques can very well complement variational  approaches. \\
In addition, Bregman distances have proven to be an interesting tool to analyze Bayesian inversion as well, which further emphasizes the strength of this concept for inverse problems theory.\\
Going from discrete to infinite dimensional Bayesian inversion has recently attracted attention \cite{LaSaSi09,HeLa09,He10,St10,DaLaStVo13} for theoretical reasons as well as for designing algorithms that work in high dimensional settings. Extending the ideas presented here to infinite dimensions could also be useful to draw further connections to variational approaches, which are often rather formulated and analyzed in a function space setting.\\
This article investigated log-concave priors, which covers many priors used in Bayesian inversion. However, especially for implementing sparsity constraints priors that do not fit into this category are used more and more often (cf. Section~\ref{subsec:SpIP}). As their use might lead to multimodal posteriors, getting more insight into the relation of the CM estimate and the local maxima of the posterior would be very valuable. \\
Connected to the last point is the extension to non-linear inverse problems.\\ 
Finally, our presentation covered Gaussian noise, only. An extension to other relevant noise models like Poisson noise would be very interesting. 


\ack
This work has been supported by the German Science Foundation (DFG) via  grant BU 2327/6-1 within the  \textit{Inverse Problems Initiative} and via Cells in Motion Cluster of Excellence (EXC 1003 CiM), University of M\"unster. The authors thank Tapio Helin, Matti Lassas and Samuli Siltanen (all Helsinki University) for stimulating this research direction.





\section*{References}

\bibliographystyle{abbrv}
\bibliography{all}

\end{document}